\documentclass{amsart}

\newcommand{\defn}[1]{\textbf{\textit{#1}}}

\usepackage{amsmath, amssymb, amsfonts}
\usepackage{amsthm}
\usepackage{mathtools}
\usepackage{microtype}
\usepackage{tikz-cd}
\usepackage{graphicx}
\usepackage[hidelinks]{hyperref}
\usepackage{enumitem}

\newcommand{\R}{\mathbb{R}} 
\newcommand{\C}{\mathbb{C}}
\renewcommand{\H}{\mathbb{H}}
\newcommand{\RP}{\mathbb{RP}}

\renewcommand{\P}{\mathcal{P}}
\newcommand{\cQ}{\mathcal{Q}}
\renewcommand{\O}{\mathcal{O}}

\newcommand{\g}{\mathfrak{g}}
\newcommand{\h}{\mathfrak{h}}
\renewcommand{\k}{\mathfrak{k}}
\newcommand{\m}{\mathfrak{m}}

\newcommand{\s}{\subseteq}
\newcommand{\ip}[1]{\langle #1\rangle}
\newcommand{\too}{\longrightarrow}
\newcommand{\mto}{\mapsto}
\newcommand{\mtoo}{\longmapsto}
\newcommand{\la}{\lambda}

\DeclareMathOperator{\Sp}{Sp}
\DeclareMathOperator{\Lie}{Lie}
\DeclareMathOperator{\Ad}{Ad}
\let\ker\relax\DeclareMathOperator{\ker}{ker}
\DeclareMathOperator{\Spec}{Spec}

\newcommand{\ps}{^{\operatorname{ps}}}
\newcommand{\ass}[1]{^{#1{\operatorname{-ss}}}}
\newcommand{\aps}[1]{^{#1{\operatorname{-ps}}}}

\DeclareFontFamily{U}{MnSymbolC}{}
\DeclareSymbolFont{MnSyC}{U}{MnSymbolC}{m}{n}
\DeclareFontShape{U}{MnSymbolC}{m}{n}{
    <-6>  MnSymbolC5
   <6-7>  MnSymbolC6
   <7-8>  MnSymbolC7
   <8-9>  MnSymbolC8
   <9-10> MnSymbolC9
  <10-12> MnSymbolC10
  <12->   MnSymbolC12}{}
\DeclareMathSymbol{\intp}{\mathbin}{MnSyC}{'270}

\newcommand{\sll}[1]{\mkern-4mu\mathbin{/\mkern-5mu/}_{\mkern-4mu{#1}}}
\newcommand{\slll}[1]{\mkern-4mu\mathbin{/\mkern-5mu/\mkern-5mu/}_{\mkern-4mu{#1}}}

\theoremstyle{plain}
\newtheorem{theorem}{Theorem}[section]
\newtheorem{lemma}[theorem]{Lemma}
\newtheorem{proposition}[theorem]{Proposition}
\newtheorem{corollary}[theorem]{Corollary}
\theoremstyle{definition}
\newtheorem{definition}[theorem]{Definition}
\newtheorem{remark}[theorem]{Remark}
\newtheorem{example}[theorem]{Example}
\numberwithin{equation}{section}

\title{Stratification of singular hyperk\"ahler
quotients}

\author{Maxence Mayrand}
\thanks{During the preparation of this work, the author was
supported by a Moussouris Scholarship from the University of
Oxford, a PGS-D and a PDF scholarship from the Natural
Sciences and Engineering Research Council of Canada (NSERC),
and a postdoctoral scholarship from the Fonds de Recherche
du Qu\'{e}bec Nature et technologies (FRQNT)}
\address{Maxence Mayrand, Department of
Mathematics, University of Toronto}
\email{mayrand@math.toronto.edu}

\begin{document}

\begin{abstract}
Hyperk\"ahler quotients by non-free actions are typically
highly singular, but are remarkably still partitioned into
smooth hyperk\"ahler manifolds. We show that these
partitions are topological stratifications, in a strong
sense. We also endow the quotients with global Poisson
structures which induce the hyperk\"ahler structures on the
strata. Finally, we give a local model which shows that
these quotients are locally isomorphic to linear
complex-symplectic reductions in the GIT sense. These
results can be thought of as the hyperk\"ahler analogues of
Sjamaar--Lerman's theorems for symplectic reduction. They
are based on a local normal form for the underlying
complex-Hamiltonian manifold, which may be of independent
interest.
\end{abstract}

%When a compact Lie group acts freely and in a Hamiltonian
%way on a symplectic manifold, the Marsden--Weinstein theorem
%says that the reduced space is a smooth symplectic manifold.
%If we drop the freeness assumption, the reduced space might
%be singular, but Sjamaar--Lerman (1991) showed that it can
%still be partitioned into smooth symplectic manifolds which
%``fit together nicely'' in the sense that they form a
%topological stratification. In this paper, we prove
%an analogue of this statement for hyperk\"ahler
%quotients construction. We also show that
%singular hyperk\"ahler quotients are complex-analytic spaces
%which are locally biholomorphic to affine complex-symplectic
%GIT quotients with biholomorphisms that are compatible with
%natural holomorphic Poisson brackets on both sides.

\maketitle

\section{Introduction}

\subsection{Overview}

Let $M$ be a hyperk\"ahler manifold and $K$ a compact Lie
group acting on $M$ by preserving the hyperk\"ahler
structure and with hyperk\"ahler moment map $\mu: M \to \k^*
\otimes \R^3$. If $K$ acts freely, then the hyperk\"ahler
quotient
\[
M \slll{} K \coloneqq \mu^{-1}(0) / K
\]
is a smooth manifold endowed with a canonical hyperk\"ahler
structure \cite{hit87}.

If the action is not necessarily free, the quotient
$\mu^{-1}(0) / K$ is typically highly singular, but is
remarkably still a union of smooth hyperk\"ahler manifolds.
This was first observed by Nakajima \cite{nak94} for quiver
varieties and, in general, by Dancer--Swann \cite{dan97}.
It is an adaptation of the work of Sjamaar--Lerman
\cite{sja91} on symplectic reductions by non-free actions,
which yields to stratified symplectic spaces.

However, not all results of Sjamaar--Lerman have been
adapted to the hyperk\"ahler setting. Stratified symplectic
spaces have much more structure than a union of symplectic
manifolds: the symplectic manifolds ``fit together nicely''
in many ways, and there is a local model generalizing the
Darboux theorem. The goal of this paper is to prove
analogues of these results for hyperk\"ahler quotients.

First, one of the main results of \cite{sja91} is that the
partition of a singular symplectic reduction into symplectic
manifolds is a topological stratification in a strong sense
(see Definition \ref{w2ev8mnf}). We will show that the same
holds for the partition of a singular hyperk\"ahler quotient
into hyperk\"ahler manifolds.

Second, Sjamaar--Lerman \cite{sja91} showed that a
singular symplectic reduction has a certain Poisson
structure which induces the symplectic structures on the
strata. Similarly, we will endow each singular hyperk\"ahler
quotient with a global structure that induces the
hyperk\"ahler structures $(g_S, \mathsf{I}_S, \mathsf{J}_S,
\mathsf{K}_S)$ on the strata $S \s M \slll{} K$. More
precisely, for each choice of frame of complex structures
$(\mathsf{I}, \mathsf{J}, \mathsf{K})$ on $M$, there is a
complex-analytic structure $\O_{\mathsf{I}}$ on $M \slll{}
K$ together with a holomorphic Poisson bracket $\{\cdot,
\cdot\} : \O_{\mathsf{I}} \times \O_{\mathsf{I}} \to
\O_{\mathsf{I}}$ such that the inclusion $S \hookrightarrow
M \slll{} K$ of each stratum is holomorphic Poisson with
respect to $\mathsf{I}_S$ and the complex-symplectic
structure $\omega_{\mathsf{J}_S} + i\omega_{\mathsf{K}_S}$
(where $\omega_{\mathsf{I}_S}, \omega_{\mathsf{J}_S},
\omega_{\mathsf{K}_S}$ are the K\"ahler forms of $S$).
Moreover, $M \slll{} K$ has the structure of a real
stratified symplectic space (in the sense of \cite{sja91})
compatible with the first K\"ahler forms
$\omega_{\mathsf{I}_S}$. That is, there is a subalgebra
$C^\infty(M \slll{} K)$ of the algebra of continuous
functions, together with a Poisson bracket, such that the
inclusions $S \hookrightarrow M \slll{} K$ are smooth and
Poisson with respect to $\omega_{\mathsf{I}_S}$.  These two
global structures on the singular space $M \slll{} K$
determine the K\"ahler forms $\omega_{\mathsf{I}_S},
\omega_{\mathsf{J}_S}, \omega_{\mathsf{K}_S}$ and hence the
whole hyperk\"ahler structures.

Third, we extend another important result of \cite{sja91},
which is that singular symplectic reductions are locally
isomorphic to linear symplectic reductions, i.e.\ reductions
of symplectic vector spaces by linear actions. This can be
viewed as a generalization of the Darboux theorem. We show
that, similarly, singular hyperk\"ahler quotients are
locally biholomorphic with respect to $\O_{\mathsf{I}}$ to
linear hyperk\"ahler quotients, in a way that is compatible
with the complex-symplectic structures
$\omega_{\mathsf{J}_S} + i\omega_{\mathsf{K}_S}$ on the
strata. These isomorphisms are not necessarily compatible
with the first K\"ahler forms $\omega_{\mathsf{I}_S}$, since
there is no Darboux theorem for hyperk\"ahler manifolds
describing all three K\"ahler forms simultaneously.

To get these results, we prove a local normal form for the
complex-Hamiltonian structure of $M$, analogous to
Guillemin--Sternberg's local normal form for a real moment
map \cite{gui82}. This result may be of independent
interest.

\subsection{Statements of results}\label{zgrptvtu}

We now give precise statements of our results; the proofs
will be in the body of the paper.

A \defn{hyperk\"ahler manifold} is a tuple $(M, g,
\mathsf{I}, \mathsf{J}, \mathsf{K})$, where $M$ is a smooth
manifold, $g$ is a Riemannian metric on $M$, and
$\mathsf{I}, \mathsf{J}, \mathsf{K}$ are three complex
structures which are K\"ahler with respect to $g$ and
satisfy $\mathsf{I}\mathsf{J}\mathsf{K} = -1$. The
corresponding K\"ahler forms will be denoted
$\omega_{\mathsf{I}}, \omega_{\mathsf{J}},
\omega_{\mathsf{K}}$. We say that a \defn{tri-Hamiltonian
hyperk\"ahler manifold} is a triple $(M, K, \mu)$, where $M$
is a hyperk\"ahler manifold, $K$ is a \emph{compact} Lie
group acting on $M$ by preserving the hyperk\"ahler
structure, and $\mu : M \to \k^*\otimes\R^3$ is a
hyperk\"ahler moment map, where $\k = \Lie(K)$.

We will also assume that the $K$-action extends to a
holomorphic $K_\C$-action with respect to any of the complex
structures (where $K_\C$ is the complexification of $K$). In
that case, we say that $(M, K, \mu)$ is \defn{integrable}
(or \defn{$\boldsymbol{I}$-integrable} if we need to specify
the complex structure). This is a natural assumption in the
context of K\"ahler or hyperk\"ahler quotients and holds in
most examples that one encounters (cf.\ Sjamaar
\cite{sja95}, Heinzner--Loose \cite{hei94}, or Kaledin
\cite{kal95}). The terminology comes from the fact that it
is equivalent to the completeness of the vector fields
$\mathsf{I}x^\#$, for $x \in \k = \Lie(K)$. For example, it
holds if $M$ is a complex affine variety and the action is
real algebraic (see e.g.\ \cite[p.\ 226]{hei00}).

%The only exceptions known to the author are obtained by
%artificially removing parts of the $K_\C$-orbits.
%(It is tempting to conjecture that it holds whenever $M$
%is geodesically complete.) 

Let $(M, K, \mu)$ be a tri-Hamiltonian hyperk\"ahler
manifold. Then, the quotient $M\slll{}K \coloneqq
\mu^{-1}(0) / K$ has a natural \defn{orbit-type partition},
whose pieces are the connected components of the subspaces
$\mu^{-1}(0)_{(H)}/K$ for all subgroups $H\s K$, where
$\mu^{-1}(0)_{(H)}$ is the set of points with stabilizer
conjugate to $H$ in $K$. These pieces are, in fact, smooth
hyperk\"ahler manifolds:

\begin{theorem}[{Dancer--Swann
\cite{dan97}}]
\label{a0rzh3o6}
Let $(M, K, \mu)$ be a tri-Hamiltonian hyperk\"ahler
manifold, let $\pi:\mu^{-1}(0)\to M\slll{}K$ be the quotient
map, and let $S\s M\slll{}K$ be an orbit-type piece. Then:

\begin{enumerate}
\item[\textup{(i)}]
$S$ is a topological manifold,
$\pi^{-1}(S)$ is a smooth submanifold of $M$, and there is a
unique smooth structure on $S$ such that $\pi^{-1}(S)\to S$
is a smooth submersion. 

\item[\textup{(ii)}]
There is a unique hyperk\"ahler structure $(g_S,
\mathsf{I}_S, \mathsf{J}_S, \mathsf{K}_S)$ on $S$ such that
the pullbacks of the K\"ahler forms $\omega_{\mathsf{I}_S},
\omega_{\mathsf{J}_S}, \omega_{\mathsf{K}_S}$ to
$\pi^{-1}(S)$ are the restrictions of the K\"ahler forms
$\omega_{\mathsf{I}}, \omega_{\mathsf{J}},
\omega_{\mathsf{K}}$ of $M$.
\end{enumerate}

\end{theorem}

Our formulation of this theorem is slightly stronger than
the one in the cited paper since we have added a
\emph{uniqueness} part in (ii), characterizing the
hyperk\"ahler structures. We will need this stronger
version, so, for completeness, we provide a full proof of
Theorem \ref{a0rzh3o6} in \S\ref{wtcqem80}.

The question of whether this partition is a topological
stratification (Definition \ref{w2ev8mnf}) was left open in
Dancer--Swann's work; even the frontier condition was not
known. The issue is that the arguments used by
Sjamaar--Lerman \cite{sja91} in the symplectic case is based
on the local normal form for the moment map
\cite{gui82,mar85}, which has no hyperk\"ahler equivalent.
Indeed, the local normal form for the moment map implies the
Darboux theorem, so a hyperk\"ahler equivalent would give a
local model describing all three symplectic forms
simultaneously and hence they could not carry any local
information. But the three symplectic forms collectively
determine the Riemannian metric, which does carry local
information: the curvature.  Nevertheless:

\begin{theorem}\label{4ztu3wpj}
Let $(M,K,\mu)$ be an integrable tri-Hamiltonian
hyperk\"ahler manifold. Then, the orbit-type partition of
$M\slll{}K$ is a topological stratification.
\end{theorem}

In fact, we will show that this partition is a Whitney
stratification (Definition \ref{9fd7wy3s}) with respect to
some complex-analytic structure.

The idea of the proof is to use the close relationship
between hyperk\"ahler geometry and complex-symplectic
geometry. Namely, $M\slll{}K$ is isomorphic to a symplectic
reduction in the category of complex-analytic spaces, and we
can adapt Sjamaar--Lerman's arguments to this setting.

More precisely, let $G\coloneqq K_\C$ and suppose, without
loss of generality, that the action is integrable with
respect to the complex structure $\mathsf{I}$. Note that we
can always arrange this by rotating $(\mathsf{I},
\mathsf{J}, \mathsf{K})$ by an element of $\mathrm{SO}(3)$.
Let $\mu_{\mathsf{I}}, \mu_{\mathsf{J}}, \mu_{\mathsf{K}}$
be the three components of the hyperk\"ahler moment map
$\mu$ and let $\mu_\R \coloneqq \mu_{\mathsf{I}}$ and
$\mu_\C \coloneqq \mu_{\mathsf{J}} + i\mu_{\mathsf{K}}$.
Then, $\mu_\C : M \to \g^*$, where $\g \coloneqq \Lie(G)$,
is a holomorphic moment map for the action of $G$ on $M$
with respect to the $\mathsf{I}$-holomorphic
complex-symplectic form $\omega_\C \coloneqq
\omega_{\mathsf{J}} + i\omega_{\mathsf{K}}$.  Moreover, by
letting 
\begin{align*}
M\ass{\mu_\R}
    &\coloneqq \{p \in M:
    \overline{G\cdot p} \cap \mu_\R^{-1}(0) \ne \emptyset
    \} \\
\mu_\C^{-1}(0)\ass{\mu_\R}
    &\coloneqq \mu_\C^{-1}(0) \cap M\ass{\mu_\R},
\end{align*}
we have $\mu^{-1}(0) \s \mu_\C^{-1}(0)\ass{\mu_\R}$ and, by
a result of Heinzner--Loose \cite{hei94}, this inclusion
descends to a homeomorphism $M\slll{}K\cong
\mu_\C^{-1}(0)\ass{\mu_\R}\sll{}G$, where $\ {\sll{}}\ $ is
a categorical quotient in the category of complex-analytic
spaces (we will review Heinzner--Loose's work in
\S\ref{r7gzqckk}).  Thus, it suffices to get a local normal
form for the complex part $\mu_\C$ of the moment map, and
this is one of the main technical results of this paper.

To state this normal form, let $p\in\mu^{-1}(0)$ and let $ V
\coloneqq (T_p(G\cdot p))^{\omega_\C}/T_p(G\cdot p)$, where
$(\cdot)^{\omega_\C}$ is the complex-symplectic complement.
Then, $V$ is a complex-symplectic vector space on which the
stabilizer $H\coloneqq G_p$ acts linearly. The normal form
says that the complex-Hamiltonian manifold $(M, \mathsf{I},
\omega_\C, G, \mu_\C)$ is completely determined in a
neighbourhood of $p$ by the representation of $H$ on $V$.
More precisely, there is a canonical structure of a
complex-Hamiltonian $G$-manifold on the associated vector
bundle $G \times_H(\h^\circ \times V)$ (see
\S\ref{s0eqcql1}), which gives the local model:

\begin{theorem}\label{vhb1qe4r}
Let $(M,K,\mu)$ be an $\mathsf{I}$-integrable
tri-Hamiltonian hyperk\"ahler manifold. Let $G\coloneqq
K_\C$, $p \in \mu^{-1}(0)$, $H \coloneqq G_p$, and $V
\coloneqq (T_p(G\cdot p))^{\omega_\C}/T_p(G\cdot p)$. Then,
there is a $G$-{\hspace{0pt}}saturated neighbourhood of $p$
in $M\ass{\mu_\R}$ which is isomorphic as a
complex-Hamiltonian $G$-manifold to a $G$-saturated
neighbourhood of $[1,0,0]$ in $G\times_H(\h^\circ\times V)$.
\end{theorem}

Here, a \defn{$\boldsymbol{G}$-saturated} subset of a
$G$-space $X$ is a subset $A$ such that $\overline{G\cdot
a}\s A$ for all $a\in A$. See Losev \cite{los06} for a
similar result in the algebraic setting. 

This local form enables us to study the structure of the
quotient:

\begin{theorem}\label{71m3l02v}
Let $(M, K,\mu)$ be an $\mathsf{I}$-integrable
tri-Hamiltonian hyperk\"ahler manifold and let $G\coloneqq
K_\C$. For each orbit-type strata $S \s M \slll{} K$, let
$(g_S, \mathsf{I}_S, \mathsf{J}_S, \mathsf{K}_S)$ be its
hyperk\"ahler structure as in \textup{Theorem
\ref{a0rzh3o6}}. Let $\mu_\R\coloneqq\mu_{\mathsf{I}}$ and
$\mu_\C\coloneqq\mu_{\mathsf{J}}+i\mu_{\mathsf{K}}$.

\begin{itemize}

\item[\textup{(i)}]
\textbf{\emph{Complex-analytic structure.}}
The inclusion $\mu^{-1}(0)\s\mu_\C^{-1}(0)\ass{\mu_\R}$
descends to a homeomorphism $M\slll{}K \cong
\mu_\C^{-1}(0)\ass{\mu_\R} \sll{} G$ and hence $M \slll{} K$
inherits the structure $\O_{\mathsf{I}}$ of a
complex-analytic space. Moreover, the orbit-type partition
is a complex-analytic Whitney stratification with respect to
$\O_{\mathsf{I}}$ and is compatible with the complex
structures $\mathsf{I}_S$ on the strata.

\item[\textup{(ii)}]
\textbf{\emph{Holomorphic Poisson structure.}}
There is a unique holomorphic Poisson bracket on
$\O_{\mathsf{I}}$ such that the inclusion of each orbit-type
piece $S\hookrightarrow M\slll{}K$ is holomorphic Poisson
with respect to $\omega_{\mathsf{J}_S} +
i\omega_{\mathsf{K}_S}$.

\item[\textup{(iii)}] 
\textbf{\emph{Real Poisson structure.}}
The first K\"ahler forms $\omega_{\mathsf{I}_S}$ are
compatible with a stratified symplectic structure in the
sense of Sjamaar--Lerman, i.e.\ there is a subalgebra
$C^\infty(M \slll{} K)$ of the algebra of real-valued
continuous functions, together with a Poisson bracket, such
that the inclusion of each orbit-type piece $S
\hookrightarrow M \slll{} K$ is smooth and Poisson with
respect to $\omega_{\mathsf{I}_S}$.

\item[\textup{(iv)}]
\textbf{\emph{Local model.}}
Let $q \in M \slll{} K$. Take a point $p \in \mu^{-1}(0)$
above $q$, let $H \coloneqq G_p$, let $V \coloneqq
(T_p(G\cdot p))^{\omega_\C}/T_p(G\cdot p)$, where $\omega_\C
\coloneqq \omega_{\mathsf{J}}+i\omega_{\mathsf{K}}$, and let
$\Phi_V:V\to\h^*$ be the moment map $\Phi_V(v)(x) =
\frac{1}{2}\omega_\C(xv,v)$. Then, $H$ is a complex
reductive group and $q$ has a neighbourhood biholomorphic
with respect to $\O_{\mathsf{I}}$ to a neighbourhood of $0$
in the GIT quotient $\Phi_V^{-1}(0) \sll{} H =
\Spec\C[\Phi^{-1}_V(0)]^H$. Moreover, this biholomorphism
respects the orbit-type stratifications and holomorphic
Poisson brackets on both sides.

\end{itemize}
\end{theorem}

\begin{remark}
(i) and (iii) imply that $M\slll{}K$ is a stratified
K\"ahler space in the sense of Huebschmann \cite[Definition
3.1]{hue11}.
\end{remark}

\begin{remark}
Using Kempf--Ness type theorems, there are many situations
where $M\slll{}K$ is isomorphic to a GIT quotient
$\mu_\C^{-1}(0)\sll{\mathcal{L}}G$ for some linearisation
$\mathcal{L}$, i.e.\ when $\mu_\C^{-1}(0)\ass{\mu_\R}$
coincides with the set of $\mathcal{L}$-semistable points.
In that case, the sheaf $\O_{\mathsf{I}}$ is simply the
underlying complex-analytic structure. For example, this is
the case for hyperk\"ahler quotients of $T^*G$ by closed
subgroups of $K \times K$ \cite{may19}, where $T^*G$ has the
hyperk\"ahler structure found by Kronheimer \cite{kro88}.
See also \cite{may17} for a study of the resulting partition
in some family of examples.
\end{remark}

\begin{remark}
The complex-symplectic GIT quotients $\Phi_V^{-1}(0)\sll{}H$
which appear as local models in this theorem have been well
studied in the literature. For example, it is known  that if
$H$ is abelian then $\Phi_V^{-1}(0)\sll{}H$ is normal
\cite{bul18}. In particular, Theorem \ref{71m3l02v}(iv)
implies that hyperk\"ahler quotients by compact tori are
normal.
\end{remark}

\subsection{Organization of the paper}

In \S\ref{prbnv29v} we give background material on
stratified spaces and reduction, in \S\ref{fzal992o} we
prove Theorem \ref{vhb1qe4r}, and in \S\ref{x09hoi74} we
prove Theorem \ref{71m3l02v}.

\subsection*{Acknowledgements}
This work was done mainly during the author's PhD at the
University of Oxford. I thank Prof.\ Andrew Dancer, my
supervisor, for his guidance.

\section{Preliminaries}\label{prbnv29v}

This section gives background material on stratified spaces,
symplectic reduction, quotients of complex-analytic spaces,
and the links between these notions. We start with a review
of the theory of topologically stratified spaces and explain
the work of Sjamaar--Lerman \cite{sja91} on singular
symplectic reduction. We then discuss links with
complex-analytic geometry, reviewing work of Heinzner--Loose
\cite{hei94} and Sjamaar \cite{sja95}. We also recall
Dancer--Swann's construction \cite{dan97} of the
hyperk\"ahler structures on the orbit-type pieces of a
singular hyperk\"ahler quotient and prove Theorem
\ref{a0rzh3o6}.

\subsection{Stratified spaces}\label{mvwurhby}

Stratified spaces are topological spaces which can be
partitioned into manifolds which ``fit together nicely''.
The underlying object for this theory is thus the following:

\begin{definition}\label{0w8kdfpkuc}
A \defn{partitioned space} is a pair $(X,\P)$ where $X$ is a
topological space and $\P$ is a partition of $X$, i.e.\ a
collection of non-empty disjoint subsets  of $X$ whose union
is $X$. The elements of $\P$ are called the \defn{pieces}.
An \defn{isomorphism} between two partitioned spaces
$(X,\P)$ and $(Y,\cQ)$ is a homeomorphism $f:X\to Y$ such
that for all $S \in \P$ there exists $T \in \cQ$ such that
$f(S) = T$.
\end{definition}

Just like manifolds are topological spaces satisfying
additional conditions, stratified spaces are partitioned
spaces with additional conditions imposed. The first step is
the following notion.

\begin{definition}[{\cite[\S1.1]{gor88}}]\label{pywimqkk}
A \defn{decomposed space} is a partitioned space $(X,\P)$
such that $X$ is second countable and Hausdorff, and the
following conditions hold:

\begin{itemize}

\item[$\bullet$] {\bf Manifold condition.}
Each piece is a topological manifold in the subspace
topology.

\item[$\bullet$] {\bf Local condition.}
$\P$ is locally finite and its pieces are locally closed.

\item[$\bullet$] {\bf Frontier condition.}
For all $S,T\in\P$, if $S\cap\overline{T}\ne\emptyset$ then
$S\s\overline{T}$.

\end{itemize}

\end{definition}

\begin{remark}
If $(X,\P)$ is a decomposed space, then there is a natural
relation on $\P$ given by $S\leq T$ if $S\s\overline{T}$. It
follows from the local closedness of the strata that this
relation is a partial order. Moreover, the frontier
condition is equivalent to
\[
\overline{S} = \bigcup_{T\le S} T,
\quad
\text{for all }S\in\P.
\]
This notion is sometimes incorporated in the definition of
decomposed space, namely we fix a poset $\mathcal{I}$ and
say that an $\mathcal{I}$-decomposed space is a
topologically stratified space $(X,\P)$ with an isomorphism
$\P\cong\mathcal{I}$ of posets.
\end{remark}

Decomposed spaces can be rather pathological: for example,
the topologist's sine curve
\begin{center}
\includegraphics[scale=0.3]{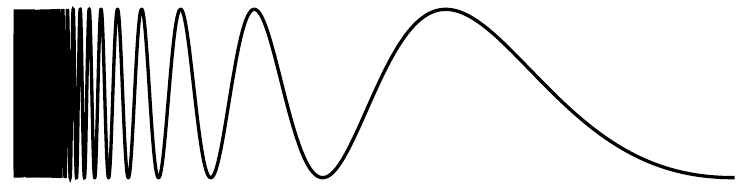}
\end{center}
is a perfectly valid one with two strata. Roughly speaking,
stratified spaces avoid such pathologies by requiring that
every point has a neighbourhood which retracts continuously
onto it. We also impose that this neighbourhood is
compatible with the partition in some sense. To make this
precise, we need a few extra notions.  First, the
\defn{dimension} of a decomposed space $(X,\P)$ is
\[
\dim(X,\P) \coloneqq \sup\{\dim S:S\in\P\}.
\]
Given two partitioned spaces $(X,\P)$ and $(Y,\cQ)$, their
\defn{cartesian product} is the partitioned space $(X\times
Y,\P\times\cQ)$ where $\P\times\cQ = \{S\times
T:S\in\P,T\in\cQ\}$. If $(X,\P)$ and $(Y,\cQ)$ are
decomposed spaces, then so is $(X\times Y,\P\times\cQ)$, and
$\dim(X\times Y,\P\times\cQ) = \dim(X,\P) + \dim(Y,\cQ)$.
Next, the \defn{cone} over a partitioned space $(X,\P)$ is
the partitioned space $(CX,C\P)$ where $CX$ is the open cone
over $X$, i.e.
\[
CX \coloneqq (X\times[0,\infty))/\{(p,0)\sim(q,0),\text{for
all }p,q\in X\}
\]
and $C\P$ is the natural partition of $CX$ given by $C\P
\coloneqq
\{S\times(0,\infty):S\in\P\}\cup\{\text{vertex}\}$. The cone
over a decomposed space $(X,\P)$ is itself a decomposed
space and has dimension $\dim(CX,C\P)=\dim(X,\P)+1$. A
topologically stratified space is defined inductively as a
decomposed space $(X,\P)$ which is locally isomorphic to
$\R^n$ times a cone over a lower-dimensional topologically
stratified space:

\begin{definition}[{\cite{gor88,sja91}}]\label{w2ev8mnf}
A zero-dimensional topologically stratified space is any
countable set of points with the discrete topology and with
any partition. A \defn{topologically stratified space} is a
finite-dimensional decomposed space $(X,\P)$ such that every
point $p\in X$ has a neighbourhood isomorphic as a
partitioned space to $\R^n\times CL$ for some $n\ge 0$ and
some compact topologically stratified space $L$, by a map
sending $p\mto\{0\}\times\{\text{vertex}\}$.
\end{definition}

For example, one-dimensional topologically stratified spaces
are locally modelled on cones over finite sets of points,
i.e.\ they are graphs:
\newcommand{\leglength}{0.5cm}
\newcommand{\dotsize}{0.5pt}
\[
\begin{tikzpicture}
\begin{scope}[shift={(-0.3, 0)}]
	\draw (0, -0.5) -- (0, 0.5);
\end{scope}
\begin{scope}[shift={(0.2, 0)}]
	\foreach\ang in {90}{
		\draw (0,0) -- (\ang:\leglength);
    }
    \node[draw, circle, fill, inner sep=\dotsize] at (0, 0) {};
\end{scope}
\begin{scope}[shift={(0.75, 0)}]
	\foreach\ang in {90, 270}{
		\draw (0,0) -- (\ang:\leglength);
    }
    \node[draw, circle, fill, inner sep=\dotsize] at (0, 0) {};
\end{scope}
\begin{scope}[shift={(1.5, 0)}]
	\foreach\ang in {90, 210, 330}{
		\draw (0,0) -- (\ang:\leglength);
    }
    \node[draw, circle, fill, inner sep=\dotsize] at (0, 0) {};
\end{scope}
\begin{scope}[shift={(2.5, 0)}]
	\foreach\ang in {0, 90, 180, 270}{
		\draw (0,0) -- (\ang:\leglength);
    }
    \node[draw, circle, fill, inner sep=\dotsize] at (0, 0) {};
    \node at (-0.35, -0.7) {one-dimensional local models};
\end{scope}
\begin{scope}[shift={(3.6, 0)}]
	\foreach\ang in {18, 90, 162, 234, 306}{
		\draw (0,0) -- (\ang:\leglength);
    }
    \node[draw, circle, fill, inner sep=\dotsize] at (0, 0) {};
\end{scope}
\begin{scope}[shift={(4.5, 0)}]
\node at (0, 0) {$\cdots$};
\end{scope}
\begin{scope}[shift={(8.5, 0)}, scale=0.4, rotate=-20]
\node[draw, circle, fill, inner sep=\dotsize] (1) at (2.442, -1.227) {};
\node[draw, circle, fill, inner sep=\dotsize] (2) at (3.662, -0.806) {};
\node[draw, circle, fill, inner sep=\dotsize] (3) at (-1.277, 1.649) {};
\node[draw, circle, fill, inner sep=\dotsize] (4) at (-2.965, -1.879) {};
\node[draw, circle, fill, inner sep=\dotsize] (5) at (-4.391, -2.505) {};
\node[draw, circle, fill, inner sep=\dotsize] (6) at (5.292, 1.340) {};
\node[draw, circle, fill, inner sep=\dotsize] (7) at (1.210, -0.583) {};
\node[draw, circle, fill, inner sep=\dotsize] (8) at (0.219, 0.657) {};
\node[draw, circle, fill, inner sep=\dotsize] (9) at (0.517, 1.805) {};
\node[draw, circle, fill, inner sep=\dotsize] (10) at (-5.460, -3.072) {};
\node[draw, circle, fill, inner sep=\dotsize] (11) at (-1.621, 2.953) {};
\node[draw, circle, fill, inner sep=\dotsize] (12) at (2.913, 2.408) {};
\node[draw, circle, fill, inner sep=\dotsize] (13) at (4.300, 0.427) {};
\node[draw, circle, fill, inner sep=\dotsize] (14) at (-2.626, -3.052) {};
\node[draw, circle, fill, inner sep=\dotsize] (15) at (-2.433, 0.918) {};
\node[draw, circle, fill, inner sep=\dotsize] (16) at (5.534, -0.211) {};
\node[draw, circle, fill, inner sep=\dotsize] (17) at (-1.730, -0.319) {};
\node[draw, circle, fill, inner sep=\dotsize] (18) at (2.504, 1.186) {};
\node[draw, circle, fill, inner sep=\dotsize] (19) at (-3.661, 1.385) {};
\node (20) at (0.5, -3.5) {a one-dimensional topologically stratified space};
\draw (1) -- (2);
\draw (5) -- (4);
\draw (17) -- (8);
\draw (13) -- (2);
\draw (8) -- (3);
\draw (10) -- (5);
\draw (14) -- (4);
\draw (15) -- (3);
\draw (4) -- (17);
\draw (3) -- (11);
\draw (13) -- (16);
\draw (8) -- (9);
\draw (12) -- (18);
\draw (13) -- (18);
\draw (8) -- (18);
\draw (6) -- (13);
\draw (1) -- (7);
\draw (15) -- (19);
\draw (7) -- (8);
\draw (15) -- (17);
\end{scope}
\end{tikzpicture}
\]
Then, two-dimensional topologically stratified spaces are
locally modelled on cones over graphs, etc. Also, manifolds
with corners are special cases.

%The compact topologically stratified space $L$ associated to
%a point $p$ in Definition \ref{w2ev8mnf} is called the
%\defn{link} at $p$ and is unique up to homeomorphisms.
%Moreover, for a connected stratum $S\in\P$, every point of
%$S$ has the same link, so we may speak of the link of the
%stratum. This is the closest notion of ``locally Euclidean''
%that we can get for partitioned spaces, namely, the local
%structure along a stratum is constant. Note that for any
%link $L$, the space $\R^n\times CL$ is contractible. In
%particular, the topologist's sine curve above is not a
%topologically stratified space.

A typical way of proving that a decomposed space $(X,\P)$ is
a topologically stratified space is by the Whitney conditions
\cite{whi65}.

\begin{definition}\label{18w9hd9y}
Let $S$ and $T$ be two disjoint smooth submanifolds of
$\R^n$. We say that $S$ is \defn{regular} over $T$ if the
following two conditions hold:

\begin{itemize}

\item[$\bullet$] \textbf{Whitney Condition A.}
If $x_i\in S$ is a sequence converging to $y\in T$ and the
sequence of subspaces $T_{x_i}S\s\R^n$ converges (in the
Grassmannian) to $V\s\R^n$, then $T_yT\s V$.

\item[$\bullet$] \textbf{Whitney Condition B.}
If $x_i\in S$ and $y_i\in T$ are two sequences converging to
$y\in T$ in such a way that that the sequence of lines
$\R(x_i-y_i)\s\R^n$ converges to $l\in\RP^{n-1}$ and the
subspaces $T_{x_i}S$ to $V\s\R^n$, then $l\s V$.

\end{itemize}

A \defn{Whitney stratification} of a subset $X$ of $\R^n$ is
a decomposition $\P$ of $X$ into smooth submanifolds of
$\R^n$ such that $S$ is regular over $T$ for all $S,T\in\P$.
\end{definition}

We have (see e.g.\  Goresky--MacPherson \cite[Ch.\ 1,
\S1.4]{gor88} or Mather \cite{mat12}):

\begin{proposition}\label{91vu0grw}
Whitney stratifications are topological stratifications in
the sense of \emph{Definition \ref{w2ev8mnf}}.\hfill\qed
\end{proposition}

Although Whitney stratifications are initially defined in
$\R^n$, the definition is purely local and is invariant
under diffeomorphisms \cite[\S2]{mat12}. In particular, it
makes sense for complex-analytic spaces:

\begin{definition}\label{9fd7wy3s}
A \defn{complex-analytic Whitney stratified space} is a
complex-analytic space $(X,\O_X)$ together with a
decomposition $\P$ of $X$ into complex submanifolds
satisfying Whitney conditions A and B.
\end{definition}

\subsection{Smooth manifold quotients}\label{kdjccddu}

Let $K$ be a compact Lie group acting smoothly on a smooth
manifold $M$. Then, the quotient $M/K$ is a topologically
stratified space with respect to a natural partition by
orbit-types. To define this partition, for each subgroup
$H\s K$, let $(H)$ be the conjugacy class of $H$ in $K$. We
say that $p\in M$ has \defn{orbit-type} $(H)$ if its
stabilizer subgroup $K_p$ is in $(H)$. Denote the set of
points of orbit-type $(H)$ by
\[
M_{(H)}\coloneqq \{p\in M:K_p\in(H)\}.
\]
The \defn{orbit-type partition} is the partition whose
pieces are the connected components of the sets $M_{(H)}/K$
for $H\s K$. This is a topological stratification as a
consequence of the slice theorem for proper group actions
(see e.g.\ \cite[Theorem 2.7.4]{dui12}).

%Recall that the slice theorem  gives a local model for the
%$K$-manifold $M$ near a point $p\in M$ in terms of $K$,
%$K_p$ and the vector space $W\coloneqq T_pM/T_p(K\cdotp p)$
%(see e.g.\ \cite[Theorem 2.4.1]{dui12}). More precisely,
%there is a $K_p$-invariant neighbourhood $U$ of $0$ in $W$
%and a $K$-equivariant diffeomorphism from $K\times_{K_p}U$
%to a neighbourhood of $p$ in $M$. This reduces
%$K$-equivariant local properties of $M$ to the
%representation theory of compact Lie groups.

\subsection{Stratified symplectic spaces}\label{xnx9ie9i}

A \defn{Hamiltonian manifold} is a triple $(M,K,\mu)$, where
$M$ is a symplectic manifold, $K$ a compact Lie group acting
on $M$ by symplectomorphisms, and $\mu:M\to\k^*$ a
$K$-equivariant moment map.  Sjamaar--Lerman \cite{sja91}
generalized the Marsden--Weinstein theorem \cite{mar74} by
showing that the orbit-type partition of
$M\sll{\mu}K\coloneqq \mu^{-1}(0)/K$ is a topological
stratification (even though $\mu^{-1}(0)$ is not a smooth
manifold) and that each piece has a canonical symplectic
structure.

Moreover, these symplectic structures are compatible with a
Poisson bracket on an appropriate substitute for the algebra
of smooth functions. More precisely, let
$C^\infty(M\sll{\mu}K)$ be $\R$-algebra of continuous
functions on $M\sll{\mu}K$ which descend from smooth
$K$-invariant functions on $M$. Then, there is a unique
Poisson bracket on $C^\infty(M\sll{\mu}K)$ such that the
inclusion of each symplectic stratum $S \hookrightarrow
M\sll{\mu}K$ is a smooth Poisson map. 

Hence, the main result of \cite{sja91} is that singular
symplectic reductions are examples of the following notion.

\begin{definition}[Sjamaar--Lerman
\cite{sja91}]\label{47d8nznm}
A \defn{stratified symplectic space} is a topologically
stratified space $(X,\P)$ with a symplectic structure on
each stratum, a subalgebra $C^\infty(X)$ of the $\R$-algebra
of continuous functions on $X$, and a Poisson bracket on
$C^\infty(X)$ such that for each stratum $S\in\P$ the
inclusion $S\hookrightarrow X$ is a Poisson map, i.e.\ for
all $f,g\in C^\infty(X)$ the restrictions $f|_S,g|_S$ are
smooth and $\{f|_S,g|_S\}= \{f, g\}|_S$.
\end{definition}

We recall the construction  of the symplectic forms on the
orbit-type pieces of $M \sll{\mu} K$ \cite[Theorem
3.5]{sja91}, as this will be useful for our discussion on
hyperk\"ahler quotients. For a closed subgroup $H\s K$, let
$M_H$ be the set of points $p\in M$ whose stabilizer is
precisely $H$. Then, the connected components of $M_H$ are
smooth symplectic submanifolds of $M$ (of possibly different
dimensions) and the group $L\coloneqq N_K(H)/H$ (where
$N_K(H)$ is the normalizer of $H$ in $K$) is compact and
acts freely on $M_H$ by preserving the symplectic forms.
Now, $\mathfrak{l}^*\coloneqq\Lie(L)^*$ can be identified
with a subspace of $\k^*$, namely, $\h^\circ\cap(\k^*)^H$,
where $\h^\circ$ is the annihilator of $\h\coloneqq\Lie(H)$
and $(\k^*)^H$ is the set of points fixed by $H$. Moreover,
if $M_H'$ denotes the union of the connected components of
$M_H$ which intersect $\mu^{-1}(0)$, then $\mu$ restricts to
a moment map $\mu_H:M_H'\to\mathfrak{l}^*$ for the action of
$L$ on $M_H'$. Since this action is free, each connected
component of $M_H\sll{\mu_H}L=\mu_H^{-1}(0)/L$ is a smooth
symplectic manifold by the standard Marsden--Weinstein
theorem \cite{mar74}. Then, the inclusion $\mu_H^{-1}(0)\s
\mu^{-1}(0)_{(H)}$ descends to a homeomorphism
$M_H\sll{\mu_H}L\cong\mu^{-1}(0)_{(H)}/K$, and this endows
each connected component of $\mu^{-1}(0)_{(H)}/K$ with a
symplectic structure. Furthermore, the pullback of each
symplectic form to the corresponding connected component of
$\mu^{-1}(0)_{(H)}$ (which is a smooth submanifold of $M$)
is the restriction of the symplectic form of $M$.

%The proof of Theorem \ref{j0n3xklh} uses the local normal
%form for the moment map of Guillemin--Sternberg \cite{gui82}
%and Marle \cite{mar85}, which is a generalization of the
%Darboux Theorem to Hamiltonian manifolds.

%In the next chapter, we will adapt
%Sjamaar--Lerman's argument to the hyperk\"ahler setting by
%proving a holomorphic version of this normal form. Thus, it
%will be useful to review the symplectic local normal form
%here.
%
%The local normal form for the moment map says that a
%Hamiltonian manifold $(M,K,\mu)$ is completely determined in
%a neighbourhood of a point $p\in \mu^{-1}(0)$ by the
%representation of $H \coloneqq K_p$ on the symplectic slice
%$V\coloneqq (T_p(K\cdot p))^{\omega}/T_p(K\cdot p)$. The
%local model is the associated vector bundle
%$K\times_H(\h^\circ\times V)$ over $K/H \cong K\cdot p$.
%This space is homeomorphic to a symplectic reduction of
%$T^*K\times V$ by $H$ and hence has a canonical symplectic
%form. Moreover, the left $K$-action
%$k\cdot[g,\xi,v]=[kg,\xi,v]$ is Hamiltonian and there is an
%explicit expression for the moment map. One shows that a
%neighbourhood of $K\cdot p$ in $M$ is isomorphic as a
%Hamiltonian $K$-manifold to a neighbourhood of the zero
%section in $K\times_H(\h^\circ\times V)$. Setting $K=1$
%recovers the Darboux theorem. Sjamaar--Lerman used this to
%prove Theorem \ref{j0n3xklh} by reducing to the case of the
%Hamiltonian manifold $K\times_H(\h^\circ\times V)$ near the
%zero section.

\subsection{K\"ahler quotients}\label{r7gzqckk}

A \defn{Hamiltonian K\"ahler manifold} is a Hamiltonian
manifold $(M,K,\mu)$ with a $K$-invariant K\"ahler structure
compatible with the symplectic form. If $K$ acts freely,
then $M\sll{\mu}K$ has a K\"ahler structure compatible with
the reduced symplectic form (see e.g.\ \cite[Theorem
3.1]{hit87}). More generally, each symplectic stratum in
Sjamaar--Lerman's stratification is K\"ahler. To see this,
it suffices to note that for each closed subgroup $H\s K$,
the space $M_H$ of points with stabiliser $H$ is now a
complex submanifold of $M$ and hence is K\"ahler.  Thus, the
connected components of $M_H\sll{\mu_H}L$ (where $\mu_H$ and
$L$ are as in \S\ref{xnx9ie9i}) are K\"ahler manifolds, and
the homeomorphism $M_H\sll{\mu_H}L\cong \mu^{-1}(0)_{(H)}/K$
gives the desired K\"ahler structures.

But we can say much more about the holomorphic aspect of
$M\sll{\mu}K$ when the action is integrable, i.e.\ extends
to a holomorphic action of $G\coloneqq K_\C$, as shown by
Sjamaar \cite{sja95} and Heinzner--Loose \cite{hei94}.
Indeed, $M\sll{\mu}K$ is homeomorphic to a categorical
quotient $M\ass{\mu}\sll{}G$ in the category of
complex-analytic spaces, or more precisely, an analytic
Hilbert quotient: the complex-analytic analogue of GIT
quotients. Good expositions can be found in
Heinzner--Huckleberry \cite{hei01,hei00} or Greb
\cite[\S2--3]{greb2010b}; we summarise the main points in
this section. See also \cite[\S2.4]{tel04}
\cite[\S2]{greb2010} \cite[\S2]{greb2015}
\cite[\S1]{greb2010b} \cite[\S0]{hei96} \cite{hei98}
\cite{hhl94} \cite{hei94}.

\subsubsection{Analytic Hilbert quotients}\label{jncd3to1}

\begin{definition}\label{q6po9pgr}
Let $(X,\O_X)$ be a complex-analytic space and $G$ a complex
reductive group acting holomorphically on $X$. An
\defn{analytic Hilbert quotient} of $X$ by $G$ is a
complex-analytic space $(Y,\O_Y)$ together with a
$G$-invariant surjective holomorphic map $\pi:X\to Y$ such
that:
\begin{itemize}
\item[(i)] the map $\pi:X\to Y$ is \defn{locally Stein},
i.e.\ $Y$ has a cover by Stein open sets whose preimages are
Stein;
\item[(ii)] $\O_Y=(\pi_*\O_X)^G$.
\end{itemize}
\end{definition}

An important consequence of this definition is that, if it
exists, an analytic Hilbert quotient is a categorial
quotient for complex-analytic spaces. In particular, it is
unique up to biholomorphisms. We denote it
\[
X\sll{}G\coloneqq \text{the analytic Hilbert quotient of $X$
by $G$ (if it exists).}
\]
Topologically, $X\sll{}G$ is the quotient of $X$ by the
equivalence relation $x\sim y$ if $\overline{G\cdot
x}\cap\overline{G\cdot y}\ne\emptyset$ and $\pi:X\to
X\sll{}G$ is the corresponding quotient map. The space
$X\sll{}G$ can also be viewed as the set of closed
$G$-orbits, i.e.\ by defining the set of \defn{polystable}
points
\[
X\ps\coloneqq \{x\in X:\text{the orbit $G\cdot x$ is closed
in $X$}\},
\]
the inclusion $X\ps\s X$ descends to a bijection $X\ps/G\to
X\sll{}G$. In particular, for every $p\in X\sll{}G$, there
is a unique closed $G$-orbit in the fibre $\pi^{-1}(p)\s X$. 

\begin{example}[GIT quotients]\label{m0r6tknt}
Let $X$ be a complex affine variety, $G$ a complex reductive
group acting algebraically on $X$, and consider the affine
GIT quotient $X\sll{}G\coloneqq\Spec \C[X]^G$ together with
the morphism $X\to X\sll{}G$ induced by the inclusion
$\C[X]^G\hookrightarrow\C[X]$. Then, the analytification of
$X\to X\sll{}G$ is an analytic Hilbert quotient
\cite[\S6.4]{hei91}. More generally, since complex affine
varieties are Stein, the analytification of any GIT quotient
is an analytic Hilbert quotient.
\end{example}

We will later need the following properties of analytic
Hilbert quotients:

\begin{proposition}\label{wfgxy42b}
Let $\pi:X\to X\sll{}G$ be an analytic Hilbert quotient. 
\begin{itemize}
\item[\textup{(i)}] An open set $U\s X$ is $G$-saturated if
and only if it is saturated with respect to $\pi$. In that
case, $U\sll{}G\coloneqq \pi(U)$ is open in $X\sll{}G$ and
the restriction $U\to U\sll{}G$ is an analytic Hilbert
quotient. 
\item[\textup{(ii)}] If $Y\s X$ is a $G$-invariant closed
complex-analytic subspace, then $Y\sll{}G\coloneqq \pi(Y)$
is a closed complex-analytic subspace of $X\sll{}G$ and the
restriction $Y\to Y\sll{}G$ is an analytic Hilbert
quotient.\hfill\qed
\end{itemize}
\end{proposition}

For (i) see \cite[\S2 Remark and \S1 Corollary]{hei98} and
for (ii) see \cite[\S1(ii)]{hei98}.

\subsubsection{The Heinzner--Loose theorem}\label{4h2x41kr}

Just as for GIT quotients, the question of existence of
analytic Hilbert quotients is a subtle one. In complete
analogy, for an action of a complex reductive group $G$ on a
complex-analytic space $X$, there does not always exist an
analytic Hilbert quotient, but in good cases, one can find a
large open subset of $X$ on which the quotient exists. For
GIT, this set depends on a choice of a linearisation, and
for analytic Hilbert quotients, it depends on a choice of a
moment map for the action of a maximal compact subgroup $K\s
G$, as we now explain. 

Let $(M,K,\mu)$ be an integrable Hamiltonian K\"ahler
manifold and let $G \coloneqq K_\C$. Define the set of
\defn{$\boldsymbol{\mu}$-semistable} points by
\[
M\ass\mu \coloneqq \{p\in M:\overline{G\cdot
p}\cap\mu^{-1}(0)\ne\emptyset\}
\]
and the set of \defn{$\boldsymbol{\mu}$-polystable} points
by
\[
M\aps\mu \coloneqq \{p\in M: G\cdot p\text{ is closed in
}M\ass\mu\}.
\]

\begin{theorem}[Heinzner--Loose \cite{hei94}]
\label{w07wd84e}
The set $M\ass\mu$ is a $G$-invariant open subset of $M$ and
the analytic Hilbert quotient $M\ass{\mu}\sll{}G$ exists.
For all $p \in M$, we have
\begin{equation}\label{ihb5mnx9}
p\in M\aps{\mu} \iff G\cdot p\cap\mu^{-1}(0)\ne\emptyset.
\end{equation}
Moreover, the inclusion $\mu^{-1}(0)\hookrightarrow
M\ass\mu$ descends to a homeomorphism $M\sll{\mu}K\to
M\ass\mu\sll{}G$. Also, for every $p\in M\aps{\mu}$ we have
$G_p=(K_p)_\C$, so $G_p$ is a complex reductive
group.\hfill\qed
\end{theorem}

\begin{remark}
Special cases of Theorem \ref{w07wd84e} were known long
before \cite{hei94}. See, for example, Guillemin--Sternberg
\cite[\S4]{gui82-geo} and Kirwan \cite[\S7.5]{kir84}. It was
also obtained independently by Sjamaar \cite{sja95} under an
additional assumption on the moment map. This result can be
thought of as an ``analytic'' version of the Kempf--Ness
theorem.
\end{remark}

\begin{remark}
Heinzner--Loose \cite{hei94} do not mention analytic Hilbert
quotients directly, but the above theorem can be deduced
from their proofs. The reformulation which we gave can be
found in Heinzner--Huckleberry \cite[\S0]{hei96}.
\end{remark}

The main ingredient in the proof of Heinzner--Loose's
theorem is the Holomorphic Slice Theorem. We briefly review
it here, since we will use it later. If $H$ is a complex Lie
subgroup of a complex Lie group $G$ and $S$ is a complex
$H$-manifold, we denote by $G\times_HS$ the quotient of
$G\times S$ by the $H$-action $h\cdot(g,x)=(gh^{-1},h\cdot
x)$. Since the $H$-action is free and proper, there is a
unique complex manifold structure on $G\times_HS$ such that
$G\times S\to G\times_HS$ is a holomorphic submersion.

\begin{definition}
Let $G$ be a complex reductive group acting holomorphically
on a complex manifold $M$. A \defn{holomorphic slice} at a
point $p$ in $M$ is a $G_p$-invariant complex submanifold
$S\s M$ containing $p$ such that $G\cdot S$ is open in $M$
and the map
\[
G\times_{G_p}S\too G\cdot S,\quad [g,x]\mtoo g\cdot x
\]
is a $G$-equivariant biholomorphism.
\end{definition}

\begin{theorem}[{Holomorphic Slice Theorem
\cite[\S2.7]{hei94} \cite[Theorem 1.12]{sja95}}]
\label{tos2xb00}
Let $(M,K,\mu)$ be an integrable Hamiltonian K\"ahler
manifold. Then, there exists a holomorphic slice at every
point $p \in M\aps\mu$.\hfill\qed
\end{theorem}

\begin{remark}
In \cite{hei94}, this is stated only for points $p\in M$
such that $\mu(p)$ is fixed by the coadjoint action, but
since $M\aps\mu=G\cdot\mu^{-1}(0)$ we deduce the above
version.
\end{remark}

This theorem enables us to study $G$-equivariant local
properties of the complex manifold $M$ near a closed orbit
of $M\ass\mu$ by the local model $G\times_{G_p}S$. This was
used by Heinzner--Loose to prove the existence of the
analytic Hilbert quotient.

\subsubsection{Stratification of analytic Hilbert
quotients}\label{am5r1mlu}

Let $\pi:X\to X\sll{}G$ be an analytic Hilbert quotient
(e.g.\ a GIT quotient). Then, as in \S\ref{kdjccddu}, the
orbit space $X\ps/G$ has a partition by $G$-orbit-types,
i.e.\ the pieces are the connected components of the sets
$(X\ps)_{(H)}/G$ for $H\s G$. Then, the bijection $X\ps/G\to
X\sll{}G$ defines a natural partition on $X\sll{}G$ which we
call the \defn{$\boldsymbol{G}$-orbit-type partition}.
Equivalently, the orbit-type of a point $p\in X\sll{}G$ is
defined to be the orbit-type of the unique closed orbit in
$\pi^{-1}(p)$.

If $(M,K,\mu)$ is a Hamiltonian K\"ahler manifold, then
$M\sll{\mu}K\cong M\ass{\mu}\sll{}G$ is an analytic Hilbert
quotient and hence has a $G$-orbit-type partition. But it
also has the $K$-orbit-type partition of Sjamaar--Lerman.
Moreover, each stratum in the $K$-orbit-type partition is a
K\"ahler manifold, and hence has a complex structure. The
next result shows that these partitions and complex
structures are the same.

\begin{theorem}[{Sjamaar \cite[Theorem
2.10]{sja95}}]\label{k1q1jx4c}\
\begin{itemize}
\item[\textup{(i)}] The homeomorphism $M\sll{\mu}K\to
M\ass{\mu}\sll{}G$ is an isomorphism of partitioned spaces. 
\item[\textup{(ii)}] The $G$-orbit-type strata of
$M\ass\mu\sll{}G$ are complex submanifolds.
\item[\textup{(iii)}] Let $S$ be a $K$-orbit-type stratum in
$M\sll{\mu}K$ and $S'$ the corresponding $G$-orbit-type
stratum in $M\ass\mu\sll{}G$. Then, the restriction $S\to
S'$ is a biholomorphism with respect to K\"ahler structure
on $S$ and the complex structure on $S'$ obtained from
\textup{(ii)}.\hfill\qed
\end{itemize}
\end{theorem}

\begin{remark} (iii) is not stated in this way in
\cite{sja95}, but is part of the
proof.
\end{remark}

\begin{remark}
As explained earlier, Sjamaar \cite{sja95} obtained
Heinzner--Loose's Theorem \ref{w07wd84e} independently, but
under an additional assumption on the moment map which he
called \textit{admissibility}. He then stated Theorem
\ref{k1q1jx4c} under the same assumption, but his proof
relies only on the validity of Theorem \ref{w07wd84e} but
not on the admissibility of the moment map.
\end{remark}

\subsection{Hyperk\"ahler quotients}\label{wtcqem80}

The goal of this section is to prove Theorem \ref{a0rzh3o6}.
We follow Dancer--Swann \cite{dan97}, refining slightly
their arguments to get the uniqueness part of the theorem.
The proof is similar to the construction of the K\"ahler
structures on the orbit-type strata of a K\"ahler quotient
explained in \S\ref{r7gzqckk}. 

We use the notation and terminology of the introduction
(\S\ref{zgrptvtu}). Let $(M, K, \mu)$ be a tri-Hamiltonian
hyperk\"ahler manifold and let $S\s M\slll{}K$ be an
orbit-type piece.  Then, $S$ is a connected component of a
set of the form $\mu^{-1}(0)_{(H)}/K$ for some closed
subgroup $H\s K$. The set $M_H$ of points with stabiliser
$H$ is now a hyperk\"ahler submanifold of $M$ and $\mu$
restricts to a hyperk\"ahler moment map
$\mu_H:M_H'\to\mathfrak{l}^*\otimes\R^3\s\k^*\otimes\R^3$
for the free action of $L\coloneqq N_K(H)/H$ on the union
$M_H'$ of the connected components of $M_H$ intersecting
$\mu^{-1}(0)$. Hence, the connected components of
$M_H\slll{}L=\mu_H^{-1}(0)/L$ are smooth hyperk\"ahler
manifolds by the usual hyperk\"ahler quotient construction
\cite[Theorem 3.2]{hit87}. Moreover, the inclusion
$\mu_H^{-1}(0)\s\mu^{-1}(0)_{(H)}$ descends to a
homeomorphism $M_H\slll{}L\to\mu^{-1}(0)_{(H)}/K$ and hence
endows each connected component of $\mu^{-1}(0)_{(H)}/K$
with a hyperk\"ahler structure. To show that this map is
indeed a homeomorphism and also to characterise the
hyperk\"ahler structures, we will need the following lemma.
This result is implicit in Sjamaar--Lerman \cite{sja91} and
Dancer--Swann \cite{dan97}, but we give a short proof for
completeness. 

\begin{lemma}
Let $K$ be a compact Lie group acting smoothly on a smooth
manifold $M$, let $H$ be a closed subgroup of $K$, and let
$L=N_K(H)/H$. Then, $M_H$ and $M_{(H)}$ are smooth
submanifolds of $M$, and the quotients $M_H/L$ and
$M_{(H)}/K$ are topological manifolds with unique smooth
structures such that the quotients maps $M_H\to M_H/L$ and
$M_{(H)}\to M_{(H)}/K$ are smooth submersions. Moreover, the
inclusion $M_H\hookrightarrow M_{(H)}$ descends to a
diffeomorphism $M_H/L\to M_{(H)}/K$.
\end{lemma}

\begin{proof}
This follows easily from the slice theorem for proper group
actions. The map $M_H/L\to M_{(H)}/K$ is bijective, so
everything reduces to local statements. Hence we may assume
(by the slice theorem) that $M=K\times_HW$ for some
representation $W$ of $H$. Then, $M_H=L\times W_H$,
$M_{(H)}=K/H\times W_H$, and $W_H$ is a linear subspace of
$W$ (the set of fixed points of $H$), so $M_H$ and $M_{(H)}$
are smooth submanifolds of $M$. Moreover, $M_H/L=W_H$ and
the quotient map $M_H\to M_H/L$ is the projection $L\times
W_H\to W_H$ and hence is a smooth submersion. Similarly, the
quotient map $M_{(H)}\to M_{(H)}/K$ is the projection
$K/H\times W_H\to W_H$. Under these identifications, the map
$M_{H}/L\to M_{(H)}/K$ is the identity map $W_H\to W_H$.
\end{proof}

\begin{proof}[Proof of Theorem \ref{a0rzh3o6}]
Let $S\s M\slll{}K$ be an orbit-type piece. Let $Z\coloneqq
\mu^{-1}(0)$ so that $S$ is a connected component of
$Z_{(H)}/K$ for some $H\s K$. As explained above, $Z_H$ is a
smooth submanifold of $M_H$ and $Z_H/L$ is a hyperk\"ahler
manifold, where $L\coloneqq N_K(H)/K$. Now, $Z_H/L$ is a
smooth submanifold of $M_H/L$ and its image under the
diffeomorphism $M_H/L\to M_{(H)}/K$ is $Z_{(H)}/K$, so the
latter is also smooth. Hence, by the transversality theorem
applied to $M_{(H)} \to M_{(H)}/K$, $Z_{(H)}$ is a smooth
submanifold of $M_{(H)}$ and $Z_{(H)}\to Z_{(H)}/K$ is a
smooth submersion. Note that $\pi^{-1}(S)$ is open in
$Z_{(H)}$, so it is also a smooth submanifold and the
restriction $\pi^{-1}(S)\to S$ is a smooth submersion.
Moreover, $\pi^{-1}(S)$ has pure dimension since $S$ is
connected and all fibres are diffeomorphic to $K/H$.

To prove the claim about the hyperk\"ahler structure, let
$\eta_{\mathsf{I}},\eta_{\mathsf{J}},\eta_{\mathsf{K}}$ be
the K\"ahler forms on $Z_{(H)}/K$ induced by the
diffeomorphism $Z_H/L\to Z_{(H)}/K$ and consider the
commutative diagram
\[
\begin{tikzcd}
Z_H \arrow[hook]{r}{i}\arrow{d}{\rho} & Z_{(H)}
\arrow{d}{\pi} \arrow[hook]{r}{j} & M\\
Z_H/L \arrow{r}{\varphi} & Z_{(H)}/K.
\end{tikzcd}
\]
We want to show that $\pi^*\eta_{\mathsf{I}} =
j^*\omega_{\mathsf{I}}$ and similarly for $\mathsf{J}$ and
$\mathsf{K}$. By construction, we have
$i^*\pi^*\eta_{\mathsf{I}} =
\rho^*\varphi^*\eta_{\mathsf{I}} =
i^*j^*\omega_{\mathsf{I}}$ and hence
$\pi^*\eta_{\mathsf{I}}$ and $j^*\omega_{\mathsf{I}}$ agree
on $T_pZ_H$ for all $p\in Z_H$. Note that since $d\varphi_p$
and $d\rho_p$ are surjective we have $T_pZ_{(H)}=T_pZ_H+\ker
d\pi_p$. Thus, to prove that $\pi^*\eta_{\mathsf{I}}$ and
$j^*\omega_{\mathsf{I}}$ agree on $T_pZ_{(H)}$ it suffices
to show that if $u\in \ker d\pi_p$ and $v\in T_pZ_{(H)}$
then
$\pi^*\eta_{\mathsf{I}}(u,v)=j^*\omega_{\mathsf{I}}(u,v)$.
Clearly, $\pi^*\eta_{\mathsf{I}}(u,v)=0$ since
$d\pi_p(u)=0$. To show that also
$j^*\omega_{\mathsf{I}}(u,v)=0$, note that $\ker
d\pi_p=T_p(K\cdot p)$ so $u=x^\#_p$ for some $x\in\k$ and
hence $\omega_{\mathsf{I}}(u,v) =
i_{x^\#}\omega_{\mathsf{I}}(v) = d\ip{\mu_{\mathsf{I}},x}(v)
= 0$ since $v\in T_pZ_{(H)}\s \ker(d\mu_{\mathsf{I}})_p$.
Hence, $\pi^*\eta_{\mathsf{I}}$ and $j^*\omega_{\mathsf{I}}$
agree on $T_pZ_{(H)}$ for all $p\in Z_H$ and since they are
$K$-invariant and $K\cdot Z_H=Z_{(H)}$ we conclude that
$\pi^*\eta_{\mathsf{I}}=j^*\omega_{\mathsf{I}}$. The same
argument also shows that
$\pi^*\eta_{\mathsf{J}}=j^*\omega_{\mathsf{J}}$ and
$\pi^*\eta_{\mathsf{K}}=j^*\omega_{\mathsf{K}}$. Since a
hyperk\"ahler structure is completely determined by its
three symplectic forms (e.g.\ $\mathsf{I} =
\omega_{\mathsf{K}}^{-1}\omega_{\mathsf{J}}$), this proves
the proposition.
\end{proof}

\section{A local normal form for the underlying
complex-Hamiltonian manifold}\label{fzal992o}

\subsection{Overview}

The goal of this section is to prove Theorem \ref{vhb1qe4r},
which establishes a local normal form for the underlying
complex-Hamiltonian manifold of a tri-Hamiltonian
hyperk\"ahler manifold analogous to the local normal form of
Guillemin--Sternberg \cite{gui82} and Marle \cite{mar85}.

Throughout this section, $(M, K, \mu)$ is an
$\mathsf{I}$-integrable tri-Hamiltonian hyperk\"ahler
manifold and $G \coloneqq K_\C$. Then, $\omega_\C\coloneqq
\omega_{\mathsf{J}}+i\omega_{\mathsf{K}}$ is a
complex-symplectic form on $(M,\mathsf{I})$ and
$\mu_\C\coloneqq\mu_{\mathsf{J}}+i\mu_{\mathsf{K}}:M\to\g^*$
is a holomorphic moment map for the action of $G$ on
$(M,\mathsf{I},\omega_\C)$ (see \cite[\S3(D)]{hit87}). We
call $(M,\mathsf{I},\omega_\C, G,\mu_\C)$ the
\defn{underlying complex-Hamiltonian manifold} of
$(M,K,\mu)$.

Let $\mu_\R\coloneqq \mu_{\mathsf{I}}$ so that
$(M,K,\mu_\R)$ is an integrable Hamiltonian K\"ahler
manifold as in \S\ref{r7gzqckk}. In particular, we have the
sets $M\ass{\mu_\R}$ and $M\aps{\mu_\R}$ as in
\S\ref{4h2x41kr}, and we will use the notations
\[
\mu_\C^{-1}(0)\ass{\mu_\R}\coloneqq \mu_\C^{-1}(0)\cap
M\ass{\mu_\R}\quad\text{and}\quad
\mu_\C^{-1}(0)\aps{\mu_\R}\coloneqq \mu_\C^{-1}(0)\cap
M\aps{\mu_\R}.
\]

The idea of the local normal form for the moment map is to
show that in a neighbourhood of a point $p\in\mu^{-1}(0)$,
the underlying complex-Hamiltonian manifold
$(M,\mathsf{I},\omega_\C,G,\mu_\C)$ is completely determined
by the representation of $H\coloneqq G_p$ on the
\defn{complex-symplectic slice}
\[
V\coloneqq(T_p(G\cdot
p))^{\omega_\C}/T_p(G\cdot p).
\]
By the Holomorphic Slice Theorem \ref{tos2xb00}, the orbit
$G\cdot p$ is embedded in $M$, so the tangent space
$T_p(G\cdot p)$ is well-defined. We have $T_p(G\cdot
p)\s\ker (d\mu_\C)_p=(T_p(G\cdot p))^{\omega_\C}$ and hence
$V$ is a well-defined complex-symplectic vector space.
Moreover, $H \coloneqq G_p = (K_p)_\C$ (by Theorem
\ref{w07wd84e}), so $H$ is a complex reductive group acting
linearly on $V$ and preserving its complex-symplectic form.
In other words, $p$ determines a complex-symplectic
representation $\rho:H\to\Sp(V,\omega_\C)$. The goal of
Theorem \ref{vhb1qe4r} is to construct a complex-Hamiltonian
manifold $E$ from $G$ and $\rho$ which is isomorphic to a
neighbourhood of $p$ in $(M,\mathsf{I},\omega_\C,G,\mu_\C)$.
The construction of $E$ is the same as the one used by
Guillemin--Sternberg \cite{gui82}, but in a
complex-symplectic setting; see also \cite[\S2]{sja91} and
\cite{los06}.

In \S\ref{s0eqcql1} we recall the constructing of the local
model $E$. In \S\ref{1k5xw5mk}, we prove a holomorphic
version of the Darboux--Weinstein theorem which we will need
for the proof of Theorem \ref{vhb1qe4r}. In
\S\ref{j9bkeg0y}, we reformulate the Holomorphic Slice
Theorem in a way that is more suitable for our purpose.
Finally, we prove Theorem \ref{vhb1qe4r} in
\S\ref{5ench1hh}.

\subsection{The local model}\label{s0eqcql1}

Let $G$ is a complex reductive group, $H$ a reductive
subgroup of $G$, and $(V,\omega_\C)$ a complex-symplectic
representation of $H$. Then, $T^*G$ has the canonical
complex-symplectic form $-d\theta$, where $\theta$ is the
tautological holomorphic 1-form. We identify $T^*G$ with
$G\times\g^*$ via left translation, i.e.\ via the
biholomorphism
\begin{equation}\label{8ir1k3h4}
G\times\g^*\too T^*G,\quad (g,\xi)\mtoo (dL_{g^{-1}})^*(\xi),
\end{equation}
where $L_{g^{-1}}:G\to G$ is left multiplication by
$g^{-1}$. Recall that a Lie group action on any manifold
lifts to a Hamiltonian action on the cotangent bundle. By
considering the action of $G\times G$ on $G$ by left and
right multiplications (i.e.\ $(a,b)\cdot g\coloneqq
agb^{-1}$) its lift to $T^*G=G\times\g^*$ is
$(a,b)\cdot(g,\xi)=(agb^{-1},\Ad_b^*\xi)$, and the moment
map is
\begin{equation}\label{d5hiyzen}
T^*G\too\g^*\times\g^*,\quad(g,\xi)\mtoo(\Ad_g^*\xi,-\xi)
\end{equation}
(see e.g.\ \cite[\S4.4]{abr78}). The representation
$H\to\Sp(V,\omega_\C)$ can also be viewed as a
complex-Hamiltonian $H$-manifold with moment map $\Phi_V: V
\to \h^*$, $\Phi_V(v)(x) = \frac{1}{2}\omega_\C(xv, v)$.
Thus, there is a Hamiltonian action of $H$ on $T^*G\times
V$, where $H$ acts on $T^*G$ as the subgroup $1\times H\s
G\times G$ and on $V$ by the given representation. Let $E$
be the complex-symplectic reduction of $T^*G\times V$ by
$H$. Since the action of $H$ on $T^*G\times V$ is free and
proper, $E$ is a complex-symplectic manifold. Moreover, the
Hamiltonian action of $G\times 1$ on $T^*G$ descends to a
Hamiltonian action of $G$ on $E$, making $E$ into a
complex-Hamiltonian $G$-manifold.

We can also rewrite $E$ in a more convenient form where the
complex moment map for the $G$-action is explicit. First,
note that the complex moment map for the action of $H$ on
$T^*G\times V$ is
\[
\la:T^*G \times V\too\h^*, \quad \la(g, \xi, v) 
= \Phi_V(v) - \xi|_\h.
\]
Take a Hermitian inner-product on $\g$ invariant under the
maximal compact subgroup $K\s G$ and let $\m$ be the
orthogonal complement to $\h$ in $\g$. This defines an
$H$-equivariant isomorphism $\h^*\cong\m^\circ\s\g^*$ so we
can view $\Phi_V$ as taking values in $\g^*$. Then, the map
\[
G\times\h^\circ\times V\too \la^{-1}(0),\quad (g,\xi,v)\mtoo
(g,\xi+\Phi_V(v),v)
\]
is a biholomorphism. The $H$-action on $\la^{-1}(0)$
corresponds to the $H$-action on $G\times\h^\circ\times V$
given by $h\cdot(g,\xi,v)=(gh^{-1},\Ad_h^*\xi,h\cdot v)$, so
$E$ is the holomorphic vector bundle
\begin{equation}\label{nsazd5mw}
E=G\times_H(\h^\circ\times V)
\end{equation}
over $G/H$. In this setup, the Hamiltonian $G$-action is
\begin{equation}\label{2dw7d96v}
G\times E\too E,\quad a\cdot[g,\xi,v]=[ag,\xi,v]
\end{equation}
and the complex moment map is
\begin{equation}\label{0qk7ej1b}
\kappa: G \times_H (\h^\circ\times V) \too \g^*,
\quad[g,\xi,v]\mtoo\Ad_g^*(\xi+\Phi_V(v)).
\end{equation}
We summarise this discussion in the following proposition.

\begin{proposition}\label{yg4qlzkl}
Let $G$ be a complex reductive group, $H$ a reductive
subgroup of $G$, and $V$ a complex-symplectic representation
of $H$. Then, the complex-symplectic manifold
\eqref{nsazd5mw} with the action \eqref{2dw7d96v} and moment
map \eqref{0qk7ej1b} is a complex-Hamiltonian manifold. \qed
\end{proposition}

\begin{remark}
Dancer--Swann \cite{dan96} showed that $E$ is a
tri-Hamiltonian hyper\-k{\"a}hler manifold whose underlying
complex-Hamiltonian manifold is the one described above.
\end{remark}

\subsection{Holomorphic Darboux--Weinstein
Theorem}\label{1k5xw5mk}

The Darboux--Weinstein theorem \cite{wei71} is a standard
result in symplectic geometry which says that if two
symplectic forms $\omega_0$ and $\omega_1$ on a manifold $M$
agree on a submanifold $N\s M$, then we can find a
diffeomorphism $f$ on a neighbourhood of $N$ such that
$f^*\omega_1=\omega_0$. There is also an equivariant version
of the theorem, where if $\omega_0$, $\omega_1$ and $N$ are
invariant under the action of a compact Lie group, then $f$
can be taken to be equivariant. By the tubular neighbourhood
theorem, it suffices to prove the result when $M$ is a
vector bundle and $N$ the zero-section, and this is indeed
how Weinstein's original proof goes \cite{wei71}. In the
holomorphic setting, there is no tubular neighbourhood
theorem, but we can still adapt Weinstein's proof to
formulate a similar statement on holomorphic vector bundles:

\begin{theorem}\label{bkd2nozc}
Let $G$ be a group acting on a holomorphic vector bundle $E$
by bundle automorphisms \textup{(}not necessarily fixing the
base\textup{)}. Let $\omega_0$ and $\omega_1$ be two
$G$-invariant complex-symplectic forms on a $G$-invariant
neighbourhood $U$ of the zero-section $Z\s E$ such that
$\omega_0|_Z=\omega_1|_Z$. Then, there are $G$-invariant
neighbourhoods $U_0$ and $U_1$ of $Z$ in $U$ and a
$G$-equivariant biholomorphism $f:U_0\to U_1$ such that
$f^*\omega_1=\omega_0$ and $f|_Z=\mathrm{Id}_Z$.
\end{theorem}

\begin{remark}
Here $\omega_i|_Z$ is the restriction of $\omega_i$ to
$(\Lambda^2T^*E)|_Z$ (this is not the same as the pullback
to $Z$).
\end{remark}

The rest of this subsection is devoted to the proof of this
theorem. Let us first briefly sketch how we will proceed.
The first step is to get a ``Poincar\'e lemma'' for the
retraction of $U$ onto $Z$, i.e.\ to construct an explicit
homotopy operator $I:\Omega^k(U)\to\Omega^{k-1}(U)$ between
the identity map and $\pi^*$, where $\pi:U\to U,v\mto 0\cdot
v$. Then, $\alpha=I(\omega_0-\omega_1)$ is a $1$-form on $U$
and, for $t$ small enough, $\omega_t\coloneqq
\omega_0+t(\omega_1-\omega_0)$ is non-degenerate, so we get
a time-dependent holomorphic vector field
$X_t=\omega_t^{-1}(\alpha)$ on a neighbourhood of $Z$. The
proof concludes by showing that the time-dependent flow of
$X$ gives a biholomorphism with the desired properties.

Let us now construct the homotopy operator. Let
$\overline{\mathbb{D}}$ be the closed unit disc centred at
$0$ in $\C$ and let $U\s E$ be as in Theorem \ref{bkd2nozc}.
By shrinking $U$ if necessary, we may assume that it is
preserved by $\overline{\mathbb{D}}$, i.e.\ $zu\in U$ for
all $z\in\overline{\mathbb{D}}$ and $u\in U$. Let
\[
W \coloneqq \{(z,u)\in\C\times U:zu\in U\}.
\]
Then, $W$ is open in $\C\times U$ and
$\overline{\mathbb{D}}\times U\s W$. Let
\[
\lambda: W\too U,\quad (z,u)\mtoo zu,
\]
and for each $z\in\overline{\mathbb{D}}$ let
\[
\iota_z: U\too W,\quad u\mtoo(z, u).
\]
Let $\Omega^k(U)$ be the space of holomorphic $k$-forms on
$U$. Then, for all $\omega\in\Omega^k(U)$, we have a family
of holomorphic $(k-1)$-forms
$\iota_z^*i_{\partial_z}\lambda^*\omega \in \Omega^{k -
1}(U)$ depending holomorphically on $z \in
\bar{\mathbb{D}}$, where $i_{\partial_z}$ is the interior
product with the vector field $\partial_z \coloneqq
\frac{\partial}{\partial z}$ on $W \s \C \times E$. Hence,
we have a linear operator
\[
I : \Omega^{k}(U) \too \Omega^{k - 1}(U), \quad
I\omega\coloneqq
\int_0^1(\iota_z^*i_{\partial_z}\lambda^*\omega)dz.
\]
Let $\pi:U\to U$ be the projection onto the zero-section.

\begin{proposition}\label{xpp5npqz}
We have $d(I\omega) + I(d\omega) = \omega - \pi^*\omega$ for
all $\omega \in \Omega^k(U)$, i.e.\ $I$ is a homotopy
operator between the identity map and $\pi^*$.
\end{proposition}

\begin{proof}
We have
\begin{align*}
d(I\omega)+I(d\omega) &= \int_0^1
\iota_z^*(di_{\partial_z}\lambda^*\omega + i_{\partial_z}
d\lambda^*\omega)dz = \int_0^1
(\iota_z^*\mathcal{L}_{\partial_z} \lambda^*\omega)dz.
\end{align*}
Now, the flow $\theta_t$ of $\partial_z$ is $\theta_t(z,u) =
(z+t,u) = \iota_{z+t}(u)$, so $\theta_t\circ \iota_0 =
\iota_t$. Hence
$\iota_t^*\mathcal{L}_{\partial_z}\lambda^*\omega =
\iota_0^*\theta_t^*\mathcal{L}_{\partial_z}\lambda^*\omega
= \iota_0^*\frac{d}{dt}\theta_t^*\lambda^*\omega = \frac{d}{dt}
\iota_0^*\theta_t^*\lambda^*\omega$, so we get
\begin{align*}
d(I\omega)+I(d\omega) &= \int_0^1
\frac{d}{dt}\iota_0^*\theta_t^*\lambda^*\omega\,dt
= \iota_0^*\theta_1^*\lambda^*\omega -
\iota_0^*\theta_0^*\lambda^*\omega = \omega-\pi^*\omega.\qedhere
\end{align*}
\end{proof}

The following observation will be useful.

\begin{lemma}\label{97cy9g5h}
Let $\omega\in\Omega^k(U)$ and let $p\in Z$. If $\omega_p=0$
then $(I\omega)_p=0$.
\end{lemma}

\begin{proof}
For all $v \in (T_pU)^{k - 1}$, we have
$(\iota_z^*i_{\partial_z}\lambda^*\omega)_p(v) =
\omega_{zp}(d\lambda(\partial_z),d\lambda(d\iota_z(v))) = 0$
since $zp=p$ as $p \in Z$. Thus,
$(\iota_z^*i_{\partial_z}\lambda^*\omega)_p=0$ for all $z$
and hence $(I\omega)_p=0$.  \end{proof}

\begin{proof}[Proof of Theorem \ref{bkd2nozc}]
Let $\eta = \omega_1 - \omega_0$ and let $\alpha=-I\eta \in
\Omega^1(U)$.  Then, $\eta=-d\alpha$ by Proposition
\ref{xpp5npqz}. Since $\eta$ is $G$-invariant, it follows
from the definition of $I$ that $\alpha$ is also
$G$-invariant. Moreover, since $\eta|_Z=0$ we have
$\alpha|_Z=0$ by Lemma \ref{97cy9g5h}.

For each $z\in \C$, define a $G$-invariant holomorphic
$2$-form on $U$ by $\omega_z=\omega_0+z\eta$. We have
$\omega_z|_Z=\omega_0|_Z$, so in particular, $\omega_z|_p$
is non-degenerate for all $(z,p)\in \C\times Z$. Let
$\mathbb{D}_r$ be the open disc of radius $r$ centred at $0$
in $\C$. By compactness of $\overline{\mathbb{D}}_2$, we can
find a neighbourhood $U'\s U$ of $Z$ such that $\omega_z|_p$
is non-degenerate for all $(z,p)\in
\overline{\mathbb{D}}_2\times U'$. Moreover, by
$G$-invariance of $\omega_z$, we can take $U'$ to be
$G$-invariant. Thus, we may assume that $\omega_z|_p$ is
non-degenerate for all
$(z,p)\in\overline{\mathbb{D}}_2\times U$. In particular,
the maps
\[
\hat{\omega}_z:TU\too T^*U,\quad v\mtoo \omega_z(v,\cdot)
\]
are vector bundle isomorphisms for all $z\in
\overline{\mathbb{D}}_2$. Define a holomorphic family of
vectors fields on $U$ by
\[
X: \mathbb{D}_2\times U\too TU,\quad (z,p)\mtoo
(\hat{\omega}_z)^{-1}(\alpha_p).
\]
Let $J=\mathbb{D}_2\cap\R=(-2,2)$ and let
$\psi:\mathcal{E}\to U$ be the smooth time-dependent flow of
the restriction $X|_{J\times U}$. That is, $\mathcal{E}$ is
the open subset of $J\times J\times M$ such that for all
$(t_0,p)\in J\times M$, the map $\psi^{(t_0,p)}(t)\coloneqq
\psi(t,t_0,p)$ is the maximally extended integral curve of
$X|_{J\times U}$ starting at $(t_0,p)$. From the general
theory of smooth time-dependent flows (see e.g.\
\cite[Theorem 9.48]{lee13}), for all $(t_1,t_0)\in J\times
J$ the set
\[
U_{(t_1,t_0)}\coloneqq \{p\in U:(t_1,t_0,p)\in\mathcal{E}\}
\]
is open, and the map
\[
\psi_{(t_1,t_0)}:U_{(t_1,t_0)}\too U_{(t_0,t_1)},\quad
p\mtoo \psi(t_1,t_0,p)
\]
is a diffeomorphism. Moreover, since $X$ is holomorphic,
$\psi_{(t_1,t_0)}$ is a biholomorphism (this follows from
the holomorphic dependence of solutions to linear systems of
ODEs on the initial conditions; see e.g.\ \cite[Ch.\ 1,
\S8]{cod55}). Since $\alpha|_Z=0$ we have $X_{(t_0,p)}=0$
for all $(t_0,p)\in J\times Z$, and hence
$\psi(t_1,t_0,p)=p$ for all $(t_1,t_0,p)\in J\times J\times
Z$. In particular, $J\times J\times Z\s \mathcal{E}$, so
$U_{(1,0)}$ and $U_{(0,1)}$ contain $Z$. We claim that the
biholomorphism $\psi_{1,0}:U_{1,0}\to U_{0,1}$ is the one we
need. First, since $\alpha$ and $\omega_z$ are
$G$-invariant, so is $X$. Hence, $U_{1,0}$ and $U_{0,1}$ are
$G$-invariant, and $\psi_{1,0}$ is $G$-equivariant.
Moreover, from \cite[Proposition 22.15]{lee13} we have for
all $t_1\in J$, 
\begin{align*}
\frac{d}{dt}\Big|_{t=t_1}\psi^*_{t,0}\omega_t &=
\psi_{t_1,0}^*\left(\mathcal{L}_{X_{t_1}}\omega_{t_1} +
\frac{d}{dt}\Big|_{t=t_1}\omega_t\right) =
\psi_{t_1,0}^*\left(i_{X_{t_1}} d\omega_t + di_{X_{t_1}}
\omega_{t_1}+\eta\right) \\
&= \psi_{t_1,0}^*(d\alpha+\eta) = 0.
\end{align*}
Thus, $\psi_{1,0}^*\omega_1=\psi_{0,0}^*\omega_0=\omega_0$. 
\end{proof}

\subsection{Linearisation of the Holomorphic Slice
Theorem}\label{j9bkeg0y}

In this subsection, we explain how to put the Holomorphic
Slice Theorem \ref{tos2xb00} in a form which will be more
convenient for our purpose. First, we want to linearise the
slice and realise neighbourhoods of orbits in $M$ as
neighbourhoods of zero-sections of vector bundles.

\begin{proposition}\label{3bo1a7fn}
Let $(M,K,\mu)$ be an integrable Hamiltonian K\"ahler
manifold, let $p\in M\aps{\mu}$, let $G\coloneqq K_\C$, let
$H\coloneqq G_p$, and let $W\coloneqq T_pM/T_p(G\cdot p)$.
Then, there is an open ball $B$ centred at $0$ in $W$, a
$G$-invariant neighbourhood $U$ of $p$ in $M$, and a
$G$-equivariant biholomorphism $G\times_H(H\cdot B)\too U$
mapping $[1,0]$ to $p$.
\end{proposition}

\begin{proof}
This is an intermediate step in Sjamaar's proof of the
Holomorphic Slice Theorem: see the top of p.\ 101 in
\cite{sja95}. It can also be proved by linearising the
action of $G_p$ on the slice $S$ at $p$ \cite[Theorem
1.21]{sja95}.
\end{proof}

It will be important to know that the open set $U$ of the
preceding proposition can be taken to be $G$-saturated.
First, we have:

\begin{proposition}\label{6yupgb8n}
Let $(M,K,\mu)$ be an integrable Hamiltonian K\"ahler
manifold and let $p\in M\aps{\mu}$. Then, every
$G$-invariant neighbourhood of $p$ contains a neighbourhood
of $p$ which is $G$-saturated in $M\ass{\mu}$.
\end{proposition}

\begin{proof}
Our argument is similar to \cite[Remark 14.24]{hei07}. As
observed in \cite[Remark 1.1]{hei98}, the quotient map
$\pi:M\ass{\mu}\to M\ass{\mu}\sll{}G$ sends
$G$-{\hspace{0pt}}invariant \allowbreak closed subsets to
closed subsets. Let $U$ be a $G$-invariant neighbourhood of
$p$ in $M\ass\mu$. Then, $C\coloneqq M\ass{\mu}-U$ is a
$G$-invariant closed subset of $M\ass{\mu}$, so $\pi(C)$ is
closed in $M\ass{\mu}\sll{}G$. Moreover, since $G\cdot p$ is
closed in $M\ass{\mu}$, we have $\pi(p)\notin \pi(C)$.
Hence, $\pi^{-1}(M\ass{\mu}\sll{}G-\pi(C))$ is a
$G$-saturated neighbourhood of $p$ contained in $U$. 
\end{proof}

The set $H\cdot B$ in Proposition \ref{3bo1a7fn} is also
$H$-saturated \cite[Corollary 4.9]{sno82} and it follows
that $G\times_H(H\cdot B)$ is $G$-saturated in $G\times_HW$.
We can then restate the Holomorphic Slice Theorem in the
following form:

\begin{theorem}\label{8w30qpuz}
Let $(M,K,\mu)$ be an integrable Hamiltonian K\"ahler manifold. Let $p\in M\aps{\mu}$, let $G\coloneqq K_\C$, let $H\coloneqq G_p$, and let $W\coloneqq T_pM/T_p(G\cdot p)$. Then, there is a $G$-saturated neighbourhood $U$ of $p$ in $M\ass\mu$, a $G$-saturated neighbourhood $U'$ of the zero-section of the vector bundle $G\times_HW$, and a $G$-equivariant biholomorphism $U'\to U$ mapping $[1,0]$ to $p$.
\end{theorem}

\begin{proof}
Let $\varphi:G\times_H(H\cdot B)\to U$ be the biholomorphism
of Proposition \ref{3bo1a7fn}. By Proposition
\ref{6yupgb8n}, there is a $G$-saturated neighbourhood $U'$
of $p$ contained in $U$. Let $B'\s B$ be an open ball
sufficiently small so that $U''\coloneqq
\varphi(G\times_H(H\cdot B'))\s U'$. Then, $U''$ is
$G$-saturated.
\end{proof}

\subsection{Proof of the Complex-Hamiltonian Local Normal
Form}\label{5ench1hh}

We now complete the proof of Theorem \ref{vhb1qe4r}. The
first step is to have an explicit expression for the
complex-symplectic form $\eta_\C$ of the local model
$E=G\times_H(\h^\circ\times V)$ at the point $q=[1,0,0]$.
Note that $G_q=H$, so $H$ acts linearly on $T_qE$. Since the
$G$-action is Hamiltonian, this is a complex-symplectic
representation of $H$ on $T_qE$. Recall that $\m\s\g$ is the
orthogonal complement to $\h$.

\begin{proposition}\label{lopjs6qu}
We have $T_qE \cong \m \times \m^*\times V$ as
complex-symplectic $H$-\hspace{0pt}representations, where
$\m\times\m^*$ has the canonical symplectic form
$((x,\varphi), (y,\psi)) \mto \psi(x) - \varphi(y)$.
Moreover, $T_q(G\cdot q)\cong\m\times 0\times 0$ under this
isomorphism.
\end{proposition}

\begin{proof}
The canonical symplectic form on $T^*G=G\times\g^*$ at
$T_{(g,\xi)}(T^*G)=\g\times\g^*$ is
\begin{equation}\label{s6aeyjso}
((x,\varphi),(y,\psi))\mtoo \psi(x)-\varphi(y)+\xi([x,y])
\end{equation}
(see e.g.\ \cite[Proposition 4.4.1]{abr78}). In particular,
if $\hat{q}\coloneqq(1,0,0)\in T^*G\times V$, the symplectic
form on $T^*G\times V$ at $T_{\hat{q}}(T^*G\times
V)=\g\times\g^*\times V$ is
\[
((x,\varphi,u),(y,\psi,v)) \mto
\psi(x)-\varphi(y)+\omega_\C(u,v).
\]
Now, we have $d\la_{\hat{q}}(x,\xi,v)=-\xi|_\h$, so the
tangent space to $\la^{-1}(0)$ at $\hat{q}$ is
$\g\times\h^\circ\times V$. Moreover,
$T_{\hat{q}}(H\cdot\hat{q})=\h\times 0\times 0$, so
\[
T_q(\la^{-1}(0)/H) =
T_{\hat{q}}\la^{-1}(0)/T_{\hat{q}}(G\cdot\hat{q}) =
\g/\h\times\h^\circ\times V.
\]
Identifying $\g/\h$ with $\m$ and $\h^\circ$ with $\m^*$
gives the result.
\end{proof}

\begin{lemma}\label{lggo05z1}
Let $H\to\Sp(R,\omega)$ be a complex-symplectic
representation and $S\s R$ an $H$-invariant isotropic
subspace. Then, $R/S\cong S^*\times S^\omega/S$ as
$H$-modules.
\end{lemma}

\begin{proof}
Let $R\to S^*$ be the composition of the isomorphism $R\to
R^*$ induced by $\omega$ with the restriction map $R^*\to
S^*$. Let $R\to S^\omega$ be the projection along the
$H$-invariant complement of $S^\omega$ in $R$ (by complete
reducibility). These maps give an $H$-equivariant surjective
map $R\to S^*\times S^\omega/S$ with kernel $S^\omega\cap
S$. Since $S$ is isotropic, we have $S^\omega\cap S=S$.
\end{proof}

\begin{proof}[Proof of Theorem \ref{vhb1qe4r}]
Since $T_p(G \cdot p) \allowbreak \cong \g/\h$ is isotropic
in $T_pM$, Lemma \ref{lggo05z1} implies that
$T_pM/T_p(G\cdot p)\cong\h^\circ\times V$, where $V$ is the
complex-symplectic slice at $p$. Thus, by Theorem
\ref{8w30qpuz}, there is a $G$-saturated neighbourhood of
$p$ in $M\ass{\mu_\R}$ which is $G$-equivariantly
biholomorphic to a $G$-saturated neighbourhood of
$q=[1,0,0]$ in $E = G\times_H(\h^\circ\times V)$. Note that
by Proposition \ref{lopjs6qu}, $T_q(G\cdot q)$ is also
isotropic with respect to the canonical complex-symplectic
form on $E$.  Note also that any $G$-invariant neighbourhood
of the zero-section $0_E=G\cdot q$ of $E$ contains a
$G$-saturated neighbourhood, namely $G\times_H(H\cdot B)$
for a sufficiently small open ball $B$. Hence, it suffices
to show that, for any two $G$-invariant complex-symplectic
forms $\omega_\C$ and $\eta_\C$ on a $G$-invariant
neighbourhood of the zero-section $0_E=G\cdot q$ in $E$ such
that $T_q0_E$ is isotropic with respect to both, there is a
$G$-equivariant biholomorphism on a possibly smaller
neighbourhood of $0_E$ which pulls back $\eta_\C$ to
$\omega_\C$. By the holomorphic version of the
Darboux--Weinstein given in Theorem \ref{bkd2nozc}, it
suffices to find such a biholomorphism that makes them match
on $0_E$. This can be reduced to a linear algebraic problem,
as we now explain. The proof is inspired from
\cite[Proposition 2]{los06}.

Since $T_q0_E \s T_qE$ is isotropic with respect to both
$\omega_\C$ and $\eta_\C$, \cite[Lemma 6]{los06} says that
there exists an $H$-equivariant linear isomorphism
$\varphi:T_qE \to T_qE$ which restricts to the identity on
$T_q0_E$ and such that $\varphi^*\eta_\C=\omega_\C$. We have
$T_qE=\m\times\m^*\times V$ and $T_q0_E=\m\times 0\times 0$,
so $\varphi$ is of the form
\[
\varphi:\m\times\m^*\times V\too \m\times\m^*\times V,\quad
\varphi(x,\xi,v)=(x+A(\xi,v),B(\xi,v)),
\]
where $A:\m^*\times V\to\m$ and $B:\m^*\times V\to\m^*\times
V$ are some linear maps, with $B$ invertible. Then,
\[
\psi:E\too E,\quad \psi([g,\xi,v])=[ge^{A(\xi,v)},B(\xi,v)]
\]
is a $G$-equivariant biholomorphism with $d\psi_q=\varphi$.
In particular, $\psi^*\eta_\C|_q=\omega_\C|_q$ and, since
$\omega_\C$ and $\eta_\C$ are $G$-invariant and $\psi$ is
$G$-equivariant, this implies that $\psi^*\eta_\C|_{g\cdot
q}=\omega_\C|_{g\cdot q}$ for all $g\in G$, i.e.\
$\psi^*\eta_\C|_{0_E}=\omega_\C|_{0_E}$.

We can now apply Theorem \ref{bkd2nozc}, which shows the
existence of a $G$-equivariant complex-symplectic
isomorphism $f:U\to U'$ such that $f(p)=q$, where $U$ is a
$G$-saturated neighbourhood of $p$ in $M\ass{\mu_\R}$ and
$U'$ a $G$-saturated neighbourhood of $q$ in $E$. It remains
to show that $\kappa\circ f=\mu_\C$.  Since $(\kappa\circ
f)(p)=0=\mu_\C(p)$ and since moment maps are unique up to a
constant (see e.g.\ \cite[Ch.\ 26]{can01}) it suffices to
show that $\kappa\circ f$ is a moment map for the $G$-action
on $M\ass{\mu_\R}$. This follows from the fact that $f$ is a
$G$-equivariant complex-symplectic isomorphism.  \end{proof}

%%%%%%%%%%%%%%%%%%%%%%%%%%%%%%%%%%%%%%%%%%%%%%%%%%

\section{Stratification of singular hyperk\"ahler
quotients}\label{x09hoi74}

The goal of this section is to prove Theorem \ref{71m3l02v},
describing the structure of singular hyperk\"ahler
quotients. Throughout this section, $(M, K, \mu)$ will be a
fixed $\mathsf{I}$-integrable tri-Hamiltonian
hyper\-k\"ahler manifold.

\subsection{Complex-analytic structure}\label{83lz4xwn}

Let us first explain how the results on analytic Hilbert
quotients of \S\ref{r7gzqckk} help us define a
complex-analytic structure on $M\slll{}K$. We use the
notation of \S\ref{fzal992o}; in particular, $G\coloneqq
K_\C$, $\mu_\C\coloneqq\mu_{\mathsf{J}}+i\mu_{\mathsf{K}}$,
and $\mu_\R\coloneqq\mu_{\mathsf{I}}$. First, note that
$\mu_\C^{-1}(0)\ass{\mu_\R}\coloneqq\mu_\C^{-1}(0)\cap
M\ass{\mu_\R}$ is a $G$-invariant closed complex-analytic
subspace of $M\ass{\mu_\R}$. Hence, by Proposition
\ref{wfgxy42b}(ii), its image
$\mu_\C^{-1}(0)\ass{\mu_\R}\sll{}G$ in
$M\ass{\mu_\R}\sll{}G$ is a closed complex-analytic
subspace, and the restriction 
\[
\mu_\C^{-1}(0)\ass{\mu_\R}\too
\mu_\C^{-1}(0)\ass{\mu_\R}\sll{}G
\]
is an analytic Hilbert quotient. Note that the space
$\mu_\C^{-1}(0)\ass{\mu_\R}\sll{}G$ has a $G$-orbit-type
partition as in \S\ref{am5r1mlu} and $M\slll{}K$ has a
$K$-orbit-type partition into hyperk\"ahler manifolds by
Theorem \ref{a0rzh3o6}. Moreover, by Heinzner--Loose's
Theorem \ref{w07wd84e} and Sjamaar's Theorem
\ref{k1q1jx4c}(i), we have $\mu^{-1}(0) \s
\mu_\C^{-1}(0)\ass{\mu_\R}$, and this inclusion descends to
an isomorphism 
\[
M\slll{}K\too\mu_\C^{-1}(0)\ass{\mu_\R}\sll{}G
\]
of partitioned spaces. In particular, $M\slll{}K$ has the
structure of a complex-analytic space. We denote the
structure sheaf by $\O_{\mathsf{I}}$.

We have shown the first part Theorem \ref{71m3l02v}(i); it
remains to show that the orbit-type partition of $M \slll{}
K$ is a complex-analytic Whitney stratification with respect
to this sheaf $\O_{\mathsf{I}}$ and is compatible with the
complex structures $\mathsf{I}_S$ on the pieces. This will
be shown later as a consequence of the local model Theorem
\ref{71m3l02v}(iv), which we show next.

\subsection{Linear hyperk\"ahler quotients}\label{kjmhibsa}

Let us first consider the case of a linear hyperk\"ahler
quotient. Let $V$ be a quaternionic vector space, i.e.\ a
real vector space endowed with three endomorphisms
$\mathsf{I,J,K}$ such that $\mathsf{I}^2 = \mathsf{J}^2 =
\mathsf{K}^2 = \mathsf{I}\mathsf{J}\mathsf{K} = -1$.  Then,
$V\cong\H^n$ for some $n$, so we may endow $V$ with a real
inner-product $\ip{\cdot,\cdot}$ such that $\mathsf{I,J,K}$
are skew-symmetric. This makes $V$ into a hyperk\"ahler
manifold, with K\"ahler forms
$\omega_{\mathsf{I}}(u,v)=\ip{\mathsf{I}u,v}$, etc. Let $L$
be a compact Lie group acting linearly on $V$ by preserving
$\ip{\cdot,\cdot}$ and $\mathsf{I,J,K}$. Then, there is a
hyperk\"ahler moment map, namely,
\[
\phi:V\too\mathfrak{l}^*\otimes\R^3,
\quad\phi(v)(x_1,x_2,x_3) =
\frac{1}{2}(\omega_{\mathsf{I}}(x_1v,
v),\omega_{\mathsf{J}}(x_2v, v),\omega_{\mathsf{K}}(x_3v,
v)).
\]
Moreover, the $L$-action extends to an
$\mathsf{I}$-complex-linear action of $H\coloneqq L_\C$, and
the underlying complex-Hamiltonian manifold is
$(V,\mathsf{I},\omega_\C,H,\Phi_V)$, where $\omega_\C
\coloneqq \omega_{\mathsf{J}}+i\omega_{\mathsf{K}}$ and
$\Phi_V$ is the canonical moment map $\Phi_V(v)(x) =
\frac{1}{2}\omega_\C(xv, v)$. By the Kempf--Ness theorem
\cite{kem79}, every point in $V$ is
$\phi_\mathsf{I}$-semistable (see e.g.\ \cite[Proposition
3.9]{nak99}), so the complex-analytic space
$(V\slll{\phi_V}L,\O_{\mathsf{I}})$ is simply the
analytification of the affine GIT quotient
$\Phi_V^{-1}(0)\sll{}H=\Spec\C[\Phi_V^{-1}(0)]^H$.

Conversely, if $H$ is any complex reductive group and
$H\to\Sp(V,\omega_\C)$ is a complex-symplectic
representation (e.g.\ a complex-symplectic slice) then
$V\cong\C^{2n}\cong\H^n$ for some $n$, so we may endow $V$
with the structure of a quaternionic vector space invariant
under the action of a maximal compact subgroup $L$ of $H$
(by averaging). Hence, the GIT quotient
$\Phi_V^{-1}(0)\sll{}H$ can always be viewed as a
hyperk\"ahler quotient.

\subsection{Local model}\label{4w2vegik}

We now prove the first part of Theorem \ref{71m3l02v}(iv)
which gives a local model for the complex-analytic structure
$\O_{\mathsf{I}}$. It will be shown later that there is a
holomorphic Poisson bracket on $\O_{\mathsf{I}}$ (Theorem
\ref{71m3l02v}(ii)) compatible with this model.

Let $q\in M\slll{}K$ and let $p\in\mu^{-1}(0)$ be a point
above $q$. Let $H\coloneqq G_p$ and let
$V\coloneqq(T_p(G\cdot p))^{\omega_\C}/T_p(G\cdot p)$ be the
complex-symplectic slice at $p$. Let $\Phi_V:V\to\h^*$ be
the canonical moment map. We want to show that $q$ has a
neighbourhood $U$ which is isomorphic as a complex-analytic
and partitioned space to a neighbourhood $U'$ of $0$ in the
GIT quotient $\Phi_V^{-1}(0)\sll{}H$. However, note that the
natural partition of $\Phi_V^{-1}(0)\sll{}H$ is by
$H$-orbit-types rather than $G$-orbit-types. To show that
the biholomorphism $U\to U'$ is an isomorphism of
partitioned spaces, we will first need to show that once we
refine the partitions into the connected components of the
pieces, the $G$-orbit-type partition of
$\Phi_V^{-1}(0)\sll{}H$ (i.e.\ saying that two points
$p,q\in\Phi_V^{-1}(0)\ps$ have the same orbit-type if $H_p$
and $H_q$ are conjugate by an element of $G$ rather than
$H$) is identical to the standard $H$-orbit-type partition.

\begin{lemma}\label{ped3pamt}
Let $(M,K,\mu)$ be a Hamiltonian K\"ahler manifold and let
$\tilde{K}$ be a compact Lie group containing $K$ as a Lie
subgroup. Then, the $K$- and $\tilde{K}$-orbit-type
partitions of $M\sll{\mu}K$ coincide. Moreover, if
$(M,K,\mu)$ is integrable, then the $\tilde{K}_\C$- and
$K_\C$-orbit-type partitions of $M\ass{\mu}\sll{}K_\C$ also
coincide.
\end{lemma}

\begin{proof}
Let $X=\mu^{-1}(0)$ and let $\pi:X\to X/K$ be the quotient
map. Let $S\s X/K$ be a $\tilde{K}$-orbit-type piece, i.e.\
a connected component of a set of the form
$X_{(H)_{\tilde{K}}}/K$ for some closed subgroup $H\s K$,
where $(H)_{\tilde{K}}$ is the conjugacy class of $H$ in
$\tilde{K}$. We have $S=\pi(T)$ for some connected component
$T$ of $X_{(H)_{\tilde{K}}}$. Fix $x\in T$. We want to show
that if $y\in T$ then $K_x$ and $K_y$ (which are conjugate
in $\tilde{K}$) are in fact conjugate in $K$. Let
\[
A\coloneqq\{y\in T:\text{$K_y$ is conjugate to $K_x$ in
$K$}\}.
\]
It suffices to show that $A$ is both open and closed in $T$.
\textit{Closed:} Let $y\in\overline{A}\cap T$ and write
$y=\lim_{n\to\infty}y_n$ with $y_n\in A$. Then, there exist
$k_n\in K$ such that $k_nK_xk_n^{-1}=K_{y_n}$ for all $n$.
Since $K$ is compact, we may assume that
$\lim_{n\to\infty}k_n=k$ for some $k\in K$. Then,
$kK_xk^{-1}\s K_y$ by continuity of the action. Moreover,
$kK_xk^{-1}$ and $K_y$ are isomorphic since they are
conjugate in $\tilde{K}$ and since they have finitely many
connected components, the inclusion $kK_xk^{-1}\s K_y$
implies that $kK_xk^{-1}=K_y$. Thus, $A$ is closed.
\textit{Open:} Let $y\in A$. By Palais \cite[Corollary 2 on
p.\ 313]{pal61} there is a neighbourhood $V$ of $y$ in $X$
such that if $z\in V$ then $K_z$ is conjugate (in $K$) to a
subgroup of $K_y$. Then, $V\cap T$ is a neighbourhood of $y$
in $T$ and $V\cap T\s A$, so $A$ is open in $T$.

The second statement amounts to show that if $H$ and $L$ are
two closed subgroups of a compact Lie group $R$, then $H$
and $L$ are conjugate in $R$ if and only if $H_\C$ and
$L_\C$ are conjugate in $R_\C$. This follows from Mostow's
decomposition, as explained by Sjamaar \cite[Proof of
Theorem 2.10, first paragraph]{sja95}.
\end{proof}

Now, by picking a quaternionic structure on the
complex-symplectic slice $V$ as explained in
\S\ref{kjmhibsa}, we can apply this result to
$(V,K_p,\phi_{\mathsf{I}})$ and infer that the $G$- and
$H$-orbit-type partitions of $\Phi_V^{-1}(0)\sll{}H$
coincide. This will be used for the last part of the
following result.

\begin{proposition}\label{u5f7q715}
Let $q\in M\slll{}K$. Take a point $p\in\mu^{-1}(0)$ above
$q$, let $H\coloneqq G_p=(K_p)_\C$, and let
$V\coloneqq(T_p(G\cdot p))^{\omega_\C}/T_p(G\cdot p)$. Then,
there is a neighbourhood $U$ of $q$ in $M\slll{}K$, an open
ball $B\s V$ around $0$, and a biholomorphism \textup{(}with
respect to $\O_{\mathsf{I}}$\textup{)} from $U$ to the image
of $(H\cdot B)\cap\Phi_V^{-1}(0)$ in the GIT quotient
$\Phi_V^{-1}(0)\sll{}H=\Spec\C[\Phi_V^{-1}(0)]^H$ which maps
$q$ to the image of $0\in\Phi_V^{-1}(0)$. Moreover, this
biholomorphism is an isomorphism of partitioned spaces. 
\end{proposition}

\begin{proof}
Let $E=G\times_H(\h^\circ\times V)$. Since $H$ is reductive
and acts freely on $G\times(\h^\circ\times V)$, $E$ is an
affine variety. Moreover, the moment map $\kappa:E\to\g^*$
is algebraic, so $\kappa^{-1}(0)$ is an affine variety in
$E$ and we can consider the GIT quotient
$\kappa^{-1}(0)\sll{}G = \Spec \C[\kappa^{-1}(0)]^G$. We
claim that $\kappa^{-1}(0)\sll{}G\cong\Phi_V^{-1}(0)\sll{}H$
as affine varieties. Indeed, we have
$\kappa^{-1}(0)=G\times_H\Phi_V^{-1}(0)$, so the inclusion
$\Phi_V^{-1}(0)\to \kappa^{-1}(0):v\mto[1,v]$ descends to a
morphism
$\psi:\Phi_V^{-1}(0)\sll{}H\to\kappa^{-1}(0)\sll{}G$. Also,
the projection
$\kappa^{-1}(0) =
G\times_H\Phi_V^{-1}(0)\to\Phi_V^{-1}(0)\sll{}H$ onto the
second factor descends to a morphism
$\kappa^{-1}(0)\sll{}G\to\Phi_V^{-1}(0)\sll{}H$ which is an
inverse of $\psi$.

Now, for an element $[g,v]\in
G\times_H\Phi_V^{-1}(0)=\kappa^{-1}(0)$ we have
$G_{[g,v]}=gH_vg^{-1}$, so $\psi$ is an isomorphism of
partitioned spaces with the $G$-orbit-type partitions on
both sides. As explained above, Lemma \ref{ped3pamt} implies
that the $G$-orbit-type partition on $\Phi_V^{-1}(0)\sll{}H$
coincides with the $H$-orbit-type partition. 

By the local normal form (Theorem \ref{vhb1qe4r}), there are
$G$-saturated neighborhoods $U\s M\ass{\mu_\R}$ and $U'\s E$
of $p$ and $[1, 0, 0]$, and an isomorphism $f:U\to U'$ of
complex-Hamiltonian $G$-manifolds. Note that $U'$ can be
taken to be of the form $U' = G\times_H(H\cdot B)$ for some
open ball $B$ around zero in $\m^*\times V$ (this is how
$U'$ was constructed in the proof). Then, $W\coloneqq
U\cap\mu_\C^{-1}(0)\ass{\mu_\R}$ is a $G$-saturated open
subset of $\mu_\C^{-1}(0)\ass{\mu_\R}$, and so is
$W'\coloneqq U'\cap\kappa^{-1}(0)$ in $\kappa^{-1}(0)$.
Moreover, by Proposition \ref{wfgxy42b}(i), the image
$W\sll{}G$ of $W$ in $\mu_\C^{-1}(0)\ass{\mu_\R}\sll{}G$ is
open and $W\to W\sll{}G$ is an analytic Hilbert quotient.
Similarly, $W'\to W'\sll{}G\s\kappa^{-1}(0)\sll{}G$ is an
analytic Hilbert quotient. Since $f:U\to U'$ is a
$G$-equivariant biholomorphism with $\kappa\circ f=\mu_\C$,
it restricts to a $G$-equivariant biholomorphism $W\to W'$
and hence to a biholomorphism $W\sll{}G\to W'\sll{}G$ which
respects the $G$-orbit-type partitions. Moreover, under the
isomorphism
$\kappa^{-1}(0)\sll{}G\cong\Phi_V^{-1}(0)\sll{}H$ above we
have an isomorphism $W'\sll{}G\cong(H\cdot
B\cap\Phi_V^{-1}(0))\sll{}H$ of complex-analytic and
partitioned spaces.
\end{proof}

\subsection{The orbit-type pieces are complex
submanifolds}\label{zevdyull}

As a first application of Proposition \ref{u5f7q715}, we
will show that the pieces in the orbit-type partition of $M
\slll{} K$ are complex submanifolds with respect to
$\O_{\mathsf{I}}$; this is one of the requirements in the
definition of complex-analytic Whitney stratifications.

We shall achieve this by describing the orbit-type partition
of $\Phi_V^{-1}(0)\sll{}H$, where $H$ is a complex reductive
group, $H\to\Sp(V,\omega_\C)$ a complex-symplectic
representation, and $\Phi_V: V \to \h^*$ the canonical
moment map. The set $V^H$ of fixed points of $H$ is a
complex-symplectic subspace, so $V = W\oplus V^H$, where $W$
is the symplectic complement. Then, $W$ is
complex-symplectic and $H$-invariant, so it provides a
complex-symplectic representation of $H$. The moment
map $\Phi_W:W\to\h^*$ associated with this representation is
simply the restriction of $\Phi_V$ to $W$, so we have the
decomposition
\[
\Phi_V^{-1}(0)\sll{}H=(\Phi_W^{-1}(0)\sll{}H)\times V^H.
\]
For each $L\s H$, let $(\Phi_W^{-1}(0)\sll{}H)_{(L)}$ be the
image of $\Phi_W^{-1}(0)\ps_{(L)}/H$ under the bijection
$\Phi_W^{-1}(0)\ps/H\to\Phi_W^{-1}(0)\sll{}H$. Then, the
pieces of the orbit-type partition of
$\Phi_V^{-1}(0)\sll{}H$ are the connected components of the
sets of the form $(\Phi_W^{-1}(0)\sll{}H)_{(L)}\times V^H$.

\begin{lemma}\label{n00l9ybu}
The orbit-type piece of $\Phi_V^{-1}(0)\sll{}H$ containing
$0$ is $\{0\}\times V^H$. 
\end{lemma}

\begin{proof}
Note that $V_{(H)} = V^H$ since if $v\in V$ and
$H_v=gHg^{-1}$ for some $g\in H$, then $gHg^{-1}\s H$, and
since $gHg^{-1}$ and $H$ are isomorphic Lie groups with
finitely many connected components this implies $gHg^{-1}=H$
and hence $H_v=H$. In particular, $W_{(H)}=W\cap V^H=0$, so
the piece containing $0$ is $(\Phi_W^{-1}(0) \sll{} H)_{(H)}
\times V^H = \{0\} \times V^H$. 
\end{proof}

\begin{proposition}\label{qzqjovnx}
The orbit-type pieces of $M\slll{}K$ are non-singular
complex-analytic subspaces with respect to
$\O_{\mathsf{I}}$.
\end{proposition}

\begin{proof}
By Lemma \ref{n00l9ybu} and Proposition \ref{u5f7q715}, the
embedding of a $K$-orbit-type piece in $M\slll{}K$ is
locally biholomorphic to the embedding of $\{0\}\times V^H$
in $(\Phi_W^{-1}(0)\sll{}H)\times V^H$.
\end{proof}

\subsection{Compatibility with the hyperk\"ahler
structures}\label{6lcwt24c}

Let $S \s M \slll{} K$ be an orbit-type piece. Then, by
Proposition \ref{qzqjovnx}, $S$ is a complex manifold. But
also, $S$ has a complex structure $\mathsf{I}_S$ as part of
its hyperk\"ahler structure $(g_S, \mathsf{I}_S,
\mathsf{J}_S, \mathsf{K}_S)$ of Theorem \ref{a0rzh3o6}. We
want to show that those are equal, or in other words:

\begin{proposition}\label{i4ja2909}
The inclusion $S \hookrightarrow M\slll{}K$ is holomorphic
with respect to $\mathsf{I}_S$ and $\O_{\mathsf{I}}$. 
\end{proposition}

\begin{proof}
We want to show that the composition $S\hookrightarrow
M\slll{}K\to \mu_\C^{-1}(0)\ass{\mu_\R}\sll{}G$ is
holomorphic, where $S$ has the complex structure
$\mathsf{I}_S$. Since $\mu_\C^{-1}(0)\ass{\mu_\R}\sll{}G$ is
a closed complex-analytic subspace of
$M\ass{\mu_\R}\sll{}G$, it suffices to show that the
composition $S\to
\mu_\C^{-1}(0)\ass{\mu_\R}\sll{}G\hookrightarrow
M\ass{\mu_\R}\sll{}G$ is holomorphic, which is the same as
the composition $S\hookrightarrow M\slll{}K\hookrightarrow
M\sll{\mu_\R}K\to M\ass{\mu_\R}\sll{}G$. The set $S$ is a
connected component of $\mu^{-1}(0)_{(H)}/K$ for some $H \s
K$. Hence, $S$ is a subset of a connected component $T$ of
$\mu_\R^{-1}(0)_{(H)}/K$. Moreover, $T$ is a stratum in the
K\"ahler quotient $M\sll{\mu_\R}K$ and, from the definition
of the K\"ahler structure on $T$ given in \S\ref{r7gzqckk}
and the definition of $\mathsf{I}_S$ given above, the
inclusion $S\hookrightarrow T$ is holomorphic. Hence, it
suffices to show that the composition $T\hookrightarrow
M\sll{\mu_\R}K\to M\ass{\mu_\R}\sll{}G$ is holomorphic, and
this follows from Theorem \ref{k1q1jx4c}(iii).
\end{proof}

\subsection{The frontier condition}\label{02umjsmv}

At this point we have shown Theorem \ref{71m3l02v}(i) except
for the fact that the orbit-type partition is a Whitney
stratification. In this section we prove the first step,
which is that this partition is a decomposition in the sense
of Definition \ref{pywimqkk} (this is a requirement in the
definition of Whitney stratified spaces). Since $K$ is
compact, $\mu^{-1}(0)/K$ satisfies the local condition, so
it only remains to show the frontier condition. This will be
achieved by the local model of Proposition \ref{u5f7q715},
so we first need to discuss how the frontier condition can
be inferred locally.

Given a partitioned space $(X,\P)$ we will denote by
$\P^\circ$ the refinement of $\P$ obtained by separating
every piece of $\P$ into its connected components. In
particular, the orbit-type partition of $M\slll{}K$ which we
are considering is the refinement $\P^\circ$ of
$\P\coloneqq\{\mu^{-1}(0)_{(K_p)}/K:p\in\mu^{-1}(0)\}$.
Also, we will say that a partitioned space $(X,\mathcal{P})$
is \defn{conical} at a stratum $S\in\P$ if
$S\s\overline{T}$ for all $T\in\P$.

The following lemma provides a local criterion for
partitioned spaces to satisfy the frontier condition.

\begin{lemma}\label{25cvqn8k}
Let $(X,\P)$ be a partitioned space. Suppose that every
point $x\in X$ has a neighbourhood $U$ such that if $S$ is
the stratum containing $x$, then $S\cap U$ is connected and
$(\P|_U)^\circ$ is conical at $S\cap U$. Then, $\P^\circ$
satisfies the frontier condition.
\end{lemma}

\begin{proof}
Let $S,T\in\P$ and let $S=\bigsqcup_i S_i$, $T=\bigsqcup_j
T_j$ be their connected components. Suppose that
$S_{i_0}\cap\overline{T_{j_0}}\ne\emptyset$ for some
$i_0,j_0$. We want to show that
$S_{i_0}\s\overline{T_{j_0}}$. The set $R\coloneqq
S_{i_0}\cap\overline{T_{j_0}}$ is closed in $S_{i_0}$, so it
suffices to show that $R$ is also open in $S_{i_0}$. Let
$x\in R$. Take a neighbourhood $U$ of $x$ in $X$ such that
$S\cap U$ is connected and $(\P|_U)^\circ$ is conical at
$S\cap U$. We claim that $S_{i_0}\cap U\s R$, or
equivalently, $S_{i_0}\cap U\s\overline{T_{j_0}}$. If $T\cap
U=\bigsqcup_k C_k$ are the connected components of $T \cap
U$, then, since $(\P|_U)^\circ$ is conical at $S\cap U$, we
have $S\cap U\s \overline{C_k}$ for all $k$. But the set of
connected components of $T\cap U$ is the union of the set of
connected components of $T_j\cap U$ for all $j$, so there
exists $k_0$ such that $C_{k_0}\s T_{j_0}\cap U$ and hence
$S_{i_0}\cap U\s S\cap
U\s\overline{C_{k_0}}\s\overline{T_{j_0}}$. 
\end{proof}

\begin{proposition}\label{47g669vy}
The orbit-type partition of $M\slll{}K$ satisfies the
frontier condition and hence is a decomposition.
\end{proposition}

\begin{proof}
Let $q\in M\slll{}K$, let $V$, $H$ and $B\s V$ be as in
Proposition \ref{u5f7q715}, and let $U=(H\cdot
B)\cap\Phi_V^{-1}(0)\sll{}H$. We denote by $[v]$ the image
of a point $v\in \Phi^{-1}_V(0)$ in the GIT quotient
$\Phi_V^{-1}(0)\sll{}H$. Then, $q$ has a neighbourhood
isomorphic to $U$ as partitioned spaces, with an isomorphism
sending $q$ to $[0]$. Let $\P$ be the orbit-type partition
of $\Phi_V^{-1}(0)\sll{}H$ and let $S\in\P$ be the piece
containing $[0]$. By Lemma \ref{25cvqn8k}, it suffices to
show that $S\cap U$ is connected and $(\P|_U)^\circ$ is
conical at $S\cap U$.  By Lemma \ref{n00l9ybu},
$S=\{[0]\}\times V^H$ so $S\cap U=\{[0]\}\times (V^H\cap B)$
is connected. To show that $(\P|_U)^\circ$ is conical at
$S\cap U$, let $T'\in(\P|_U)^\circ$. Then, $T'$ is a
connected component of $T\cap U$, where
$T\coloneqq(\Phi_W^{-1}(0)\sll{}H)_{(L)}\times V^H$ for some
$L\s H$. We need to show that $S\cap U\s\overline{T'}$. Let
$([0],v)\in S\cap U$, where $v\in V^H\cap B$. Take any point
$([w],u)$ of $T'$, where $w\in(\Phi_W^{-1}(0)\ps)_{(L)}$,
$u\in V^H$, and $w+u\in H\cdot B$. It suffices to find a
continuous path $\gamma:(0,1]\to T\cap U$ such that
$\gamma(1)=([w],u)$ and $\lim_{t\to 0}\gamma(t)=([0],v)$.
Let $h\in H$ be such that $w+u\in h^{-1}\cdot B$. Then,
$hw+u\in B$. We also have $v\in B$, so there exists $t_0>0$
small enough so that $t_0hw+v\in B$ and hence $t_0w+v\in
H\cdot B$. Now, $\Phi_W(tw)=t^2\Phi_W(w)=0$ and hence
$([tw],v)\in T\cap U$ for all $t>0$ and $([tw],v)\to
([0],v)$ as $t\to 0$. Moreover, since $B$ is convex, the
straight line from $t_0w+v$ to $w+u$ will stay in $(H\cdot
B)\cap((\Phi_W^{-1}(0)\ps)_{(L)}\times V^H)$ and hence
$([t_0w],v)$ and $([w],u)$ are in the same path component
$T'$ of $T\cap U$.
\end{proof}

\subsection{Whitney conditions}\label{vpj7myf2}

We show that the orbit-type partition of $M\slll{}K$ is a
complex-analytic Whitney stratification with respect to
$\O_{\mathsf{I}}$ and hence a stratification in the sense of
Definition \ref{w2ev8mnf}. In particular, this completes
the proof of Theorem \ref{71m3l02v}(i). Our proof is similar
to that of Sjamaar--Lerman \cite[\S6]{sja91}. Let us first
recall the following result of Whitney.

\begin{lemma}[{Whitney \cite[Lemma
19.3]{whi65}}]\label{rnwtmgk1}
Let $S$ and $T$ be disjoint complex submanifolds of a
complex-analytic space $X$ with $S\s\overline{T}$ and $\dim
S<\dim T$. There is a (possibly empty) complex-analytic
subspace $A$ of $S$ with $\dim A<\dim S$ such that $T$ is
regular over $S\setminus A$.\qed
\end{lemma}

\begin{corollary}\label{qbagt6ko}
Let $X$ be a complex-analytic space and $T\s X$ a complex
submanifold with $\dim T>0$. Then, $T$ is regular over
$\{x\}$ for all $x\in \overline{T}\setminus T$.
\end{corollary}

\begin{proof}
Use Lemma \ref{rnwtmgk1} with $S=\{x\}$.
\end{proof}

\begin{proposition}
The orbit-type partition of $M\slll{}K$ is a
complex-analytic Whitney stratification with respect to
$\O_{\mathsf{I}}$. In particular, it is a stratification in
the sense of \emph{Definition \ref{w2ev8mnf}}.
\end{proposition}

\begin{proof}
By Proposition \ref{u5f7q715}, the problem reduces to
checking Whitney conditions for the $H$-orbit-type partition
of $\Phi_V^{-1}(0)\sll{}H$ at $[0]$. By \S\ref{zevdyull}, we
have $\Phi_V^{-1}(0)\sll{}H=(\Phi_W^{-1}(0)\sll{}H)\times
V^H$ and hence it suffices to check Whitney condition for
$\Phi_W^{-1}(0)\sll{}H$ at $[0]$. But the piece containing
$[0]$ is the singleton $\{[0]\}$, so this follows from
Corollary \ref{qbagt6ko}.
\end{proof}

\subsection{Poisson structure}\label{ir3fzjkd}

We now show Theorem \ref{71m3l02v}(ii), which says that
there is a natural Poisson bracket on $\O_{\mathsf{I}}$
making $M\slll{}K$ a stratified symplectic space as in
Sjamaar--Lerman's work (\S\ref{xnx9ie9i}) but in a
complex-analytic sense.

The definition of the Poisson bracket on $\O_{\mathsf{I}}$
is as follows. Let $U\s M\slll{}K$ be open, let
$f,g\in\O_{\mathsf{I}}(U)$ and let $q\in U$. To define
$\{f,g\}(q)$, let $S\s M\slll{}K$ be the orbit-type stratum
containing $q$ and let $(g_S,
\mathsf{I}_S,\mathsf{J}_S,\mathsf{K}_S)$ be its
hyperk\"ahler structure. Then, $(\omega_S)_\C \coloneqq
\omega_{\mathsf{J}_S}+i\omega_{\mathsf{K}_S}$ is a
complex-symplectic form on $(S,\mathsf{I}_S)$. By
Proposition \ref{i4ja2909}, the restrictions $f|_{S\cap U}$,
$g|_{S\cap U}$ are $\mathsf{I}_S$-holomorphic, and hence we
can take their Poisson bracket $\{f|_{S\cap U},g|_{S\cap
U}\}:S\cap U\to\C$ with respect to $(\omega_S)_\C$ and
define $\{f,g\}(q)\coloneqq\{f|_{S\cap U},g|_{S\cap
U}\}(q)$. This defines a function $\{f,g\}:U\to\C$
pointwise, and the goal is to show that it is holomorphic,
i.e.\ $\{f,g\}\in\O_{\mathsf{I}}(U)$. 

In what follows, we identify $S$ with a $G$-orbit-type
stratum in $\mu_\C^{-1}(0)\ass{\mu_\R}\sll{}G$, i.e.\ $S$ is
a connected component of
$(\mu_\C^{-1}(0)\ass{\mu_\R}\sll{}G)_{(H)}$ for some
reductive subgroup $H\s G$. By the definition of the
$G$-orbit-type partition, the map
$(\mu_\C^{-1}(0)\aps{\mu_\R})_{(H)}\to
(\mu_\C^{-1}(0)\ass{\mu_\R}\sll{}G)_{(H)}$ is surjective
(note that on the left-hand side we use polystable points),
so $S$ is the image under the quotient map
$\mu_\C^{-1}(0)\ass{\mu_\R} \to
\mu_\C^{-1}(0)\ass{\mu_\R}\sll{}G$ of an open subset $Z$ of
$(\mu_\C^{-1}(0)\aps{\mu_\R})_{(H)}$.

\begin{lemma}\label{wuem7guw}
The set $Z$ is a complex submanifold of $M$, the map
$\pi:Z\to S$ is a holomorphic submersion, and
$\pi^*(\omega_S)_\C=i^*\omega_\C$ where $i:Z\hookrightarrow
M$.
\end{lemma}

\begin{proof}
By the local normal form (Theorem \ref{vhb1qe4r}), the
embedding of $Z$ in $M$ is locally biholomorphic to the
embedding of $G/H\times V^{H}$ in
$G\times_{H}(\h^\circ\times V)$ and $\pi$ is locally
biholomorphic to the projection $G/H\times V^{H}\to V^{H}$.
This proves the first and second assertions. For the third
assertion, we first note that, since the pullbacks of the
symplectic forms $\omega_{\mathsf{I}_S},
\omega_{\mathsf{J}_S}, \omega_{\mathsf{K}_S}$ on
$\mu^{-1}(0)_{(H)}$ are the restrictions of the symplectic
forms $\omega_{\mathsf{I}}, \omega_{\mathsf{J}},
\omega_{\mathsf{K}}$ on $M$, we have
$j^*(\pi^*(\omega_S)_\C) = j^*(i^*\omega_\C)$ where
$j:\mu^{-1}(0)_{(H)}\hookrightarrow Z$. Since $j$ descends
to a diffeomorphism $\mu^{-1}(0)_{(H)}/K\to
(\mu_\C^{-1}(0)\ass{\mu_\R}\sll{}G)_{(H)}$ we get that for
all $p\in \mu^{-1}(0)_{(H)}$,
$T_pZ=T_p\mu^{-1}(0)_{(H)}+T_p(G\cdot p)$. Hence, the result
follows by the same argument as in the proof of Theorem
\ref{a0rzh3o6} given in \S\ref{wtcqem80}.
\end{proof}

\begin{lemma}
Let $f:U\to\C$ be a holomorphic $G$-invariant function on an
open set $U\s M$, and let $\Xi_f$ be the holomorphic vector
field on $U$ dual to $df$ under $\omega_\C$. Then, $\Xi_f$
is tangent to $Z$, i.e.\ $\Xi_f(p)\in T_pZ$ for all $p\in
Z\cap U$.
\end{lemma}

\begin{proof}
Let $\m=\h^\perp$ as in \S\ref{s0eqcql1}. By the local
normal form we may assume that $M=G\times_H(\m^*\times V)$,
$p=[1,0,0]$ and $Z=G/H\times V^H$. By Lemma \ref{lopjs6qu},
$T_pM=\m\times\m^*\times V$, $Z=\m\times 0\times V^H$, and
$T_p(G\cdot p)=\m\times 0\times 0$. Let
$(x,\xi,v)\coloneqq\Xi_f(p)\in \m\times\m^*\times V$. Then,
$df_p(y,\eta,w)=\eta(x)-\xi(y)+\omega_\C(v,w)$ for all
$(y,\eta,w)\in \m\times\m^*\times V$. Since $f$ is
$G$-invariant, we have $df_p(\m\times 0\times 0)=0$, so
$\xi=0$. Also, $G$-equivariance implies that for all $w\in
V$ and $h\in H$ we have $df_p(0,0,h\cdot w)=df_p(0,0,w)$, so
$\omega_\C(v,h\cdot w)=\omega_\C(v,w)$. Since $\omega_\C$ is
$H$-invariant, this implies $\omega_\C(h^{-1}v-v,w)=0$ for
all $w\in V$ and $h\in H$, so $v\in V^H$. Thus,
$\Xi_f(p)=(y,0,v)\in \m\times 0\times V^H=T_pZ$. 
\end{proof}

\begin{lemma}\label{utxo38xn}
For all open set $U\s M\slll{}K$ and
$f,g\in\O_{\mathsf{I}}(U)$, we have
$\{f,g\}\in\O_{\mathsf{I}}(U)$.
\end{lemma}

\begin{proof}
We identify $M\slll{}K$ with $\mu_\C^{-1}(0)\ass{\mu_\R}
\sll{} G$. Let $\Pi: \mu_\C^{-1}(0)\ass{\mu_\R} \to
\mu_\C^{-1}(0)\ass{\mu_\R} \sll{} G$ be the quotient map.
Then, $\{f,g\}\in\O_{\mathsf{I}}(U)$ if and only if the
pullback $\Pi^*\{f,g\}:\Pi^{-1}(U)\to\C$ is holomorphic.
This is a local statement, so we may assume that
$\Pi^{-1}(U)=\mu_\C^{-1}(0)\ass{\mu_\R}\cap U'$ for some
$G$-invariant open set $U'\s M\ass{\mu_\R}$ such that
$\Pi^*f$ and $\Pi^*g$ extend to holomorphic $G$-invariant
functions $\hat{f},\hat{g}:U'\to\C$. Then, it suffices to
show that $\Pi^*\{f,g\}=\{\hat{f},\hat{g}\}|_{\Pi^{-1}(U)}$.
Since $\hat{f}$, $\hat{g}$ and $\omega_\C$ are
$G$-invariant, so is $\{\hat{f},\hat{g}\}$. Thus, it
suffices to show that
$\Pi^*\{f,g\}(p)=\{\hat{f},\hat{g}\}(p)$ for every
polystable point $p\in
\Pi^{-1}(U)\cap\mu_\C^{-1}(0)\aps{\mu_\R}$. We have $p\in Z$
for some $Z$ as above. Let $S=\Pi(Z)$, $\pi=\Pi|_Z:Z\to S$
and $i:Z\to M$, as before. Then, we have
$d\pi(\Xi_{\hat{f}}(p))=\Xi_f(\pi(p))$, where $\Xi_f$ is the
Hamiltonian vector field of $f$ on $U\cap S$, since for all
$v\in T_pZ$,
\begin{align*}
(\omega_S)_\C(d\pi(\Xi_{\hat{f}}(p)),d\pi(v)) &=
\omega_\C(\Xi_{\hat{f}}(p),v) =
d\hat{f}_p(v)=df_{\pi(p)}(d\pi(v)) \\
&=(\omega_S)_\C(\Xi_f(\pi(p)),d\pi(v)).
\end{align*}
Thus,
\begin{align*}
\{f,g\}(\Pi(p)) &\coloneqq
(\omega_S)_\C(\Xi_{f}(\pi(p)),\Xi_{f}(\pi(p))) =
(\omega_S)_\C(d\pi(\Xi_{\hat{f}}(p)),d\pi(\Xi_{\hat{f}}(p)))
\\
&= \omega_\C(\Xi_{\hat{f}}(p),\Xi_{\hat{g}}(p))
= \{\hat{f},\hat{g}\}(p).
\end{align*}
So $\Pi^*\{f,g\}=\{\hat{f},\hat{g}\}|_{\Pi^{-1}(U)}$ and
hence $\{f,g\}\in\O_{\mathsf{I}}(U)$.
\end{proof}

By construction, the Poisson bracket is uniquely determined
by the property that the inclusions of the strata are
Poisson maps. Thus, we have show Theorem \ref{71m3l02v}(ii).

\subsection{Compatibility of the local
model}\label{gx5jsq9w}

We now show the remaining part of Theorem
\ref{71m3l02v}(iv), which is that the local model is
compatible with the holomorphic Poisson bracket constructed
in the previous section.

Let $H$ be a complex reductive group and
$H\to\Sp(V,\omega_\C)$ a complex-{\hspace{0pt}}symplectic
representation. Then, as explained in \S\ref{kjmhibsa}, we
can view the affine GIT quotient
$V_0\coloneqq\Phi_V^{-1}(0)\sll{}H$ as a hyperk\"ahler
quotient. Hence, if $\O_{V_0}$ denotes the underlying
complex-analytic structure of $V_0$, then $(V_0,\O_{V_0})$
together with the $H$-orbit-type partition is a
complex-analytic Whitney stratified space with a holomorphic
Poisson bracket (which does not depend on the choice of
quaternionic structure). Recall from Proposition
\ref{u5f7q715} that $\Phi_V^{-1}(0)\sll{}H$ provides a local
model for the complex-analytic structure of $M\slll{}K$.
Here we show that $\Phi_V^{-1}(0)\sll{}H$ is also a local
model for the Poisson structure.

\begin{proposition}
The biholomorphism of \textup{Proposition \ref{u5f7q715}} is
compatible with the holomorphic Poisson brackets.
\end{proposition}

\begin{proof}
Since the local normal form for $(M,K,\mu)$ is an
isomorphism of complex-symplectic manifolds, we only need to
show that the isomorphism $\kappa^{-1}(0)\sll{}G =
\Phi_V^{-1}(0)\sll{}H$ of affine varieties in the proof of
Proposition \ref{u5f7q715} respects the Poisson brackets.
This follows from the fact that $V$ is a complex-symplectic
submanifold of $E$ via the embedding $\iota:V\hookrightarrow
G\times_H(\h^\circ\times V)$, $v\mto [1,0,v]$ and that the
isomorphism descends from this map.
\end{proof}

\subsection{Real Poisson structure}

We now prove Theorem \ref{71m3l02v}(iii), i.e.\ we show that
$M \slll{} K$ has the structure of a stratified symplectic
space (Definition \ref{47d8nznm}) compatible with the first
K\"ahler forms.

Let $C^\infty(M \slll{} K)$ be the subalgebra of the
$\R$-algebra of continuous functions on $M \slll{} K$
consisting functions descending from smooth $K$-invariant
functions on $M$. In other words, $f \in C^\infty(M \slll{}
K)$ if and only if there exists $F \in C^\infty(M)^K$ such
that $\pi^*f = F|_{\mu^{-1}(0)}$, where $\pi : \mu^{-1}(0)
\to M \slll{} K$ is the quotient map.

\begin{lemma}
The inclusion $S \hookrightarrow M \slll{} K$ of an
orbit-type stratum is a smooth map, i.e. for all $f \in
C^\infty(M \slll{} K)$, $f|_S$ is a smooth function on $S$.
\end{lemma}

\begin{proof}
Let $F \in C^\infty(M)^K$ be such that $f \circ \pi =
F|_{\mu^{-1}(0)}$. Recall that, by Theorem \ref{a0rzh3o6},
$\pi^{-1}(S)$ is a smooth submanifold of $M$ and the
restriction $\pi^{-1}(S) \to S$ is a surjective submersion.
Hence, $F|_{\pi^{-1}(S)}$ descends to a unique smooth
function $S \to \R$, which is just $f|_S$.
\end{proof}

Hence, we can define a Poisson bracket pointwise as in
\cite{sja91} by letting
\[
\{f, g\}(x) \coloneqq \{f|_{S},
g|_{S}\}_{\omega_{\mathsf{I}_S}}(x),
\]
where $S$ is the unique orbit-type stratum containing $x$
and $\{\cdot, \cdot\}_{\omega_{\mathsf{I}_S}}$ is the real
Poisson bracket on $S$ induced by $\omega_{\mathsf{I}_{S}}$.
It only remains to show that $\{f, g\} \in C^\infty(M
\slll{} K)$ (the Leibniz rule and Jacobi identity
follow from that of $\{\cdot,
\cdot\}_{\omega_{\mathsf{I}_S}}$).

\begin{proposition}
For all $f, g \in C^\infty(M \slll{} K)$ we have $\{f, g\}
\in C^\infty(M \slll{} K)$.
\end{proposition}

\begin{proof}
First note that $(M, K, \mu_\R)$ is a
Hamiltonian manifold as in \S\ref{xnx9ie9i}, so
Sjamaar--Lerman's original theorem holds, i.e.\ $M
\sll{\mu_\R} K$ is endowed with a Poisson $\R$-algebra
$C^\infty(M \sll{\mu_\R} K)$ defined analogously. By
definition, the inclusion $M \slll{} K \hookrightarrow M
\sll{\mu_\R} K$ is smooth with respect to $C^\infty(M
\slll{} K)$ and $C^\infty(M \sll{\mu_\R} K)$.

Now, let $f, g \in C^\infty(M \slll{} K)$, so that $f \circ
\pi = F|_{\mu^{-1}(0)}$ and $g \circ \pi = G|_{\mu^{-1}(0)}$
for some $F, G \in C^\infty(M)^K$. Let $\tilde{\pi} :
\mu_\R^{-1}(0) \to M \sll{\mu_\R} K$ be the quotient map.
Then, $F|_{\mu_\R^{-1}(0)} = \tilde{f} \circ \tilde{\pi}$
and $G|_{\mu_\R^{-1}(0)} = \tilde{g} \circ \tilde{\pi}$, for
some $\tilde{f}, \tilde{g} \in C^\infty(M \sll{\mu_\R} K)$.
Note that $\tilde{f}|_{M \slll{} K} = f$ and $\tilde{g}|_{M
\slll{} K} = g$.  Let $x \in M \slll{} K$, let $S$ be the
stratum containing $x$, and let $\tilde{S}$ be the stratum
of $M \sll{\mu_\R} K$ containing $x$. Then, by construction
of the hyperk\"ahler structure on $S$ (see
\S\ref{wtcqem80}), $S$ is a K\"ahler submanifold of
$\tilde{S}$. In particular, $\{f, g\}(x) = \{f|_S,
g|_S\}_S(x) = \{\tilde{f}|_{\tilde{S}},
\tilde{g}|_{\tilde{S}}\}_{\tilde{S}}(x) = \{\tilde{f},
\tilde{g}\}(x)$. Hence, $\{f, g\} = \{\tilde{f},
\tilde{g}\}|_{M \slll{} K} \in C^\infty(M \slll{} K)$ since
$\{\tilde{f}, \tilde{g}\} \in C^\infty(M \sll{\mu_\R} K)$.
\end{proof}

\end{document}